\documentclass[11pt]{amsart}

\usepackage{algorithm}
\usepackage{algpseudocode}
\usepackage{epigraph}
\usepackage{mathtools}
\usepackage{latexsym}
\usepackage{mathrsfs}
\usepackage{pstricks}
\usepackage{listings}
\usepackage{hyperref}
\usepackage{pgfplots}
\usepackage{appendix}
\hypersetup{
  unicode=true,
  pdfauthor={},
  pdftitle={Periodic/subharmonic solutions in second order difference equations:
            existence and bifurcation results},
  pdfsubject={},
  pdfkeywords={}; {}
}
\usepackage{pst-all}	%call the pstricks package
\usepackage{mathdots}
\usepackage{xcolor}
\usepackage{enumitem}
\usepackage{todonotes}
\usepackage{amssymb,bm}
\usepackage{amsmath}
\usepackage{amsfonts}
\usepackage{amsthm}
\usepackage{tikz}
\usepackage{epsfig}
\usepackage{verbatim}
\usepackage{float}
\usepackage{tabularx}
\usepackage{bm}
\usepackage{mathtools}
\usepackage{blkarray, bigstrut}
\usepackage{subcaption}

\definecolor{codegray}{rgb}{0.5,0.5,0.5}
\lstdefinestyle{mystyle}{
    commentstyle=\color{codegreen},
    keywordstyle=\color{magenta},
    numberstyle=\tiny\color{codegray},
    stringstyle=\color{codepurple},
    basicstyle=\ttfamily\footnotesize,
    breakatwhitespace=false,         
    breaklines=true,                 
    captionpos=b,                    
    keepspaces=true,                 
    numbers=left,                    
    numbersep=5pt,                  
    showspaces=false,                
    showstringspaces=false,
    showtabs=false,                  
    tabsize=2
}
  \lstdefinelanguage{GAP}{
    basicstyle=\ttfamily,
    keywords={true, false, function, return, fail, if, in, while, do, od, else, elif, fi, break, continue},
    keywordstyle=\color{blue}\bfseries,
    otherkeywords={% Operators
      >, <, ==
    },
    breaklines=true,      
    identifierstyle=\color{black},
    sensitive=True,
    comment=[l]{\#},
    commentstyle=\color{cyan},
    stringstyle=\color{red},
    morestring=[b]',
    morestring=[b]"
  }

\providecommand{\U}[1]{\protect\rule{.1in}{.1in}}

\newcolumntype{Y}{>{\raggedleft\arraybackslash}X}
\def\bc{{\mathbb{C}}}
\def\bq{{\mathbb{Q}}}
\def\bn{{\mathbb{N}}}

\def\br{{\mathbb{R}}}

\def\bz{{\mathbb{Z}}}

\def\br{\mathbb R}

\def\vs{\vskip.3cm}
\def\noi{\noindent}

\def\gdeg{G\text{\rm -deg}}
\def\Gammadeg{\bz_2^2 \times \Gamma \text{\rm -deg}}
\def\s1deg{S^1\text{\rm -deg}}

\def\Om{\Omega}

\def\sign{\text{\rm sign\,}}

\DeclareMathOperator{\id}{Id}

\newcommand\cV{\ensuremath{\mathcal V}}

\newrgbcolor{violet}{.6 .1 .8}
  \newrgbcolor{lightyellow}{1 1 .8}
  \newrgbcolor{lightblue}{.80 1 1}
  \newrgbcolor{mygreen}{0 .66 .05}
  \definecolor{mygreen}{rgb}{0,.66,.05}
  \definecolor{lightyellow}{rgb}{1,1,.80}
  \newrgbcolor{orange}{1 .6 0}
  \newrgbcolor{GreenYellow}{.85 1 .31}
  \newrgbcolor{Yellow}{1  1  0}
  \newrgbcolor{Goldenrod}{1  .90  .16}
  \newrgbcolor{Dandelion}{1  .71  .16}
  \newrgbcolor{Apricot}{1  .68  .48}
  \newrgbcolor{Peach}{1  .50  .30}
  \newrgbcolor{Melon}{1  .54  .50}
  \newrgbcolor{YellowOrange}{1  .58  0}
  \newrgbcolor{Orange}{1  .39  .13}
  \newrgbcolor{BurntOrange}{1  .49  0}
  \newrgbcolor{Bittersweet}{1.  .4300  .24}
  \newrgbcolor{RedOrange}{1  .23  .13}
  \newrgbcolor{Mahogany}{1.  .4475  .4345}
  \newrgbcolor{Maroon}{1.  .4084  .5376}
  \newrgbcolor{BrickRed}{1.  .3592  .3232}
  \newrgbcolor{Red}{1  0  0}
  \newrgbcolor{OrangeRed}{1  0  .50}
  \newrgbcolor{RubineRed}{1  0  .87}
  \newrgbcolor{WildStrawberry}{1  .04  .61}
  \newrgbcolor{CarnationPink}{1  .37  1}
  \newrgbcolor{Salmon}{1  .47  .62}
  \newrgbcolor{Magenta}{1  0  1}
  \newrgbcolor{VioletRed}{1  .19  1}
  \newrgbcolor{Rhodamine}{1  .18  1}
  \newrgbcolor{Mulberry}{.6668  .1180  1.}
  \newrgbcolor{RedViolet}{.9538  .4060  1.}
  \newrgbcolor{Fuchsia}{.5676  .1628  1.}
  \newrgbcolor{Lavender}{1  .52  1}
  \newrgbcolor{Thistle}{.88  .41  1}
  \newrgbcolor{Orchid}{.68  .36  1}
  \newrgbcolor{DarkOrchid}{.60  .20  .80}
  \newrgbcolor{Purple}{.55  .14  1}
  \newrgbcolor{Plum}{.50  0  1}
  \newrgbcolor{Violet}{.98 .15 .95}
  \newrgbcolor{RoyalPurple}{.25  .10  1}
  \newrgbcolor{BlueViolet}{.84  .38  .98}
  \newrgbcolor{Periwinkle}{.43  .45  1}
  \newrgbcolor{CadetBlue}{.38  .43  .77}
  \newrgbcolor{CornflowerBlue}{.35  .87  1}
  \newrgbcolor{MidnightBlue}{.4414  .9259  1.}
  \newrgbcolor{NavyBlue}{.06  .46  1}
  \newrgbcolor{RoyalBlue}{0  .50  1}
  \newrgbcolor{Blue}{0  0  1}
  \newrgbcolor{Cerulean}{.06  .89  1}
  \newrgbcolor{Cyan}{0  1  1}
  \newrgbcolor{ProcessBlue}{.04  1  1}
  \newrgbcolor{SkyBlue}{.38  1  .88}
  \newrgbcolor{Turquoise}{.15  1  .80}
  \newrgbcolor{TealBlue}{.1572  1.  .6668}
  \newrgbcolor{Aquamarine}{.18  1  .70}
  \newrgbcolor{BlueGreen}{.15  1  .67}
  \newrgbcolor{Emerald}{0  1  .50}
  \newrgbcolor{JungleGreen}{.01  1  .48}
  \newrgbcolor{SeaGreen}{.31  1  .50}
  \newrgbcolor{Green}{0  1  0}
  \newrgbcolor{ForestGreen}{.1992  1.  .2256}
  \newrgbcolor{PineGreen}{.3100  1.  .5575}
  \newrgbcolor{LimeGreen}{.50  1  0}
  \newrgbcolor{YellowGreen}{.56  1  .26}
  \newrgbcolor{SpringGreen}{.74  1  .24}
  \newrgbcolor{OliveGreen}{.6160  1.  .4300}
  \newrgbcolor{RawSienna}{.53  .28  .16}
  \newrgbcolor{Sepia}{1.  .7510  .70}
  \newrgbcolor{Brown}{.41  .25  .18}
  \newrgbcolor{TAN}{.86  .58  .44}
  \newrgbcolor{Gray}{1.  1.  1.}
  \newrgbcolor{Black}{1  1  1}
  \newrgbcolor{White}{1  1  1}

\newcommand{\Mod}[1]{\ (\mathrm{mod}\ #1)}

\newtheorem{theorem}{Theorem}[section]
\newtheorem{proposition}{Proposition}[section]
\newtheorem{lemma}{Lemma}[section]
\newtheorem{corollary}{Corollary}[section]

\newtheorem{definition}{Definition}[section]

\newtheorem{remark}{Remark}[section]

\newtheorem{remark-definition}{Remark and Definition}[section]
\newtheorem{rem-not}{Remark and Notation}[section]

\begin{document}

\title[Global Bifurcation of Time-Dependent Solutions]{Global Bifurcation in Symmetric Systems of Nonlinear Wave Equations}

\author{Carlos Garcia-Azpeitia}\address{Departamento de Matem\'aticas y Mec\'anica, IIMAS-UNAM, Apdo. Postal 20-126, Col. San \'Angel,
Mexico City, 01000,  Mexico}
\email{cgazpe@mym.iimas.unam.mx}

\author{ Ziad Ghanem}\address{Department of Mathematical Sciences, University of Texas at Dallas, Richardson, TX 75080, USA}
\email{Ziad.Ghanem@UTDallas.edu}

\author{Wieslaw Krawcewicz}\address{Department of Mathematical Sciences, University of Texas at Dallas, Richardson, TX 75080, USA}
\email{Wieslaw.Krawcewicz@UTDallas.edu}

\date{}

\maketitle

\begin{abstract}
In this paper, we use the equivariant  degree theory to establish a global bifurcation result for the existence of non-stationary branches of solutions to a nonlinear, two-parameter family of hyperbolic wave equations with local delay and non-trivial damping. As a motivating example, we consider an application of our result to a system of $N$ identical vibrating strings with dihedral coupling relations.
\end{abstract}

\noi \textbf{Mathematics Subject Classification:} Primary: 37G40, 35B06; Secondary: 47H11, 35J57,  35J91

\medskip

\noi \textbf{Key Words and Phrases:} nonlinear wave equation, symmetric bifurcation, unbounded branches,  non-stationary solutions, $S^1$-equivariant degree.

%---

\section{Introduction} \label{sec:intro}
Let $\Gamma$ be a finite group,  $V := \br^N$ an orthogonal $\Gamma$-representation and denote by $L^\Gamma(V)$ the space of all $\Gamma$-equivariant linear transformations of $V$. We consider the following parameterized by $(\alpha, \beta) \in \br \times \br$ system of delayed, hyperbolic equations
\begin{equation}\label{eq:system}
\begin{cases}
\nu^2 \partial^2_t u - \partial^2_x u + \delta \partial_t u + \beta S_\tau u = f(u) - A(\alpha)u, \quad u(t,x) \in V; \\
u(t,-\frac{\pi}{2})= u(t,\frac{\pi}{2}) = 0, \quad x \in [-\frac{\pi}{2},\frac{\pi}{2}], \; t \geq 0; \\
u(t + 2\pi, x) = u(t,x) ,
\end{cases}
\end{equation}
where $(\nu,\delta,\tau) \in \br \times \br \times \br^+$ are fixed parameters, $S_\tau u(t,x) := u(t-\tau,x)$ is the $\tau$-delay operator, $A: \br \rightarrow L^\Gamma(V)$ is a continuous family of $\Gamma$-equivariant matrices and $f: V \rightarrow V$ is a differentiable function satisfying the conditions:
\vs
  \begin{enumerate}[label=($A_\arabic*$)]
\item\label{a1} $f(-u) = -f(u)$ for all $u \in V$; 
\item\label{a2} $\sigma f(u) = f(\sigma u)$ for all $u \in V, \; \sigma \in \Gamma$;
%\item\label{a3} $f(u) = o(u)$, that is
%\[
%\lim\limits_{u \rightarrow 0} \frac{| f(u) |}{| u |} = 0;
%\]
\item\label{a3} there exist $a,b > 0$ and $\eta_0,\eta_1 \geq 2$ such that
\begin{align*}
    \begin{cases}
    | f(u) | \leq a | u |^{\eta_0/2}; \\
    | f'(u) | \leq b | u |^{\eta_1/2}.
\end{cases}
\end{align*}
\end{enumerate}
In the context of the most natural application for \eqref{eq:system} --- that of an arrangement of
$N$ coupled identical vibrating strings subjected to nonlinear forces, non-trivial damping, and environmental feedback --- the vector $u(t,x) = (u_1(t,x), u_2(t,x), \ldots, u_N(t,x))$ represents the vertical displacement of our strings at any time $t \geq 0$ and position $x$ along the interval $[-\frac{\pi}{2},\frac{\pi}{2}]$, the bifurcation parameters $\alpha$ and $\beta$ are related to the coupling strength and the feedback intensity, respectively,
the {\it wave frequency} $\nu$ depends on the material properties (such as density, mass, tension), the {\it damping coefficient} $\delta$ indicates the total frictional forces influencing the strings and the {\it delay variable }$\tau$ is understood as the amount of time needed for vibrational outputs to feedback into the system.
\vs
Notice that the boundary value problem \eqref{eq:system} admits the zero solution for every parameter pair $(\alpha,\beta) \in \br \times \br$. The {\it nontrivial} functions which satisfy \eqref{eq:system} for particular parameter pairs can be categorized as follows: 
\begin{itemize}
    \item {\it stationary solutions}, depending only on the position variable $x \in [-\frac{\pi}{2},\frac{\pi}{2}]$ and which, though more interesting than the trivial solution, can often be identified using classical methods (e.g. by reducing \eqref{eq:system} to a system of ODEs);
    \item {\it non-stationary solutions}, expressing dependency on the time variable $t \in \br^+$ and which may exhibit exotic spatio-temporal symmetries that enhance our understanding of the model.
\end{itemize}
Our primary objectives are $(\rm i)$ to demonstrate how equivariant degree theory can be used in order to identify parameter values $(\alpha_0,\beta_0) \in \br \times \br$ at which branches of non-trivial periodic solutions to \eqref{eq:system} emerge from the zero solution, $(\rm ii)$ to determine 
under what conditions these branches consist only of non-stationary solutions and $(\rm iii)$ to describe their possible global and symmetric properties. Although models similar to \eqref{eq:system} have undergone
treatment as local bifurcation problems (cf. \cite{Koso1,Koso2,Koso3}), to the best of our knowledge, our present study is the first time that the global bifurcation of solutions to a nonlinear wave equation with non-trivial damping and delay has been addressed.
\vs
A prerequisite for the application of any degree-theory based argument to global bifurcation in \eqref{eq:system} is the reformulation of our system as a fixed point equation in an appropriate functional space $\mathscr H$ with a nonlinear operator in the form of a compact perturbation of identity. In order to appreciate the difficulty of this task, consider the differential operator $\mathscr F: \br \times \br \times \mathscr H \rightarrow \mathscr H$ given by
\[
\mathscr F(\alpha,\beta,u):= u - \mathscr L^{-1}(f(u) - A(\alpha)u + u - \beta S_\tau u),
\]
where 
\[
\mathscr L: \mathscr H \rightarrow \mathscr H, \quad \mathscr L:= \nu^2 \partial_t^2 - \partial_x^2 + \delta \partial_t + \id,
\]
and notice that \eqref{eq:system} is equivalent to the operator equation 
\[
\mathscr F(\alpha,\beta,u) = 0.
\]
The main complication arises from the fact that the existence, continuity and compactness of the inverse operator $\mathscr L^{-1}$ depends on the model variables $\nu$ and $\delta$. For example, it can be shown that $\delta > 0$ is a sufficient condition for the continuity of $\mathscr L^{-1}$. Indeed, this is the setting Kosovalić et al. consider in \cite{Koso1,Koso2,Koso3}, where Lyapanuv-Schmidt procedures, based on the implicit function theorem, are implemented to study the local Hopf bifurcation of periodic solutions for a model similar to \eqref{eq:system}. 
\vs
It can easily be verified that $\mathscr L^{-1}$ is also continuous under the conditions $\delta = 0$, $\nu \in \bq$, in which case, local bifurcation can be approached with techniques similar to those used to prove the existence of periodic solutions, also called free vibrations, in conservative hyperbolic equations (cf. \cite{AmZe80,Ar,Ki79,Ki00,Ra78,Ga17,Ga19,GaLe20,Br78}). When $\delta = 0$ and $\nu$ belongs to a particular set of exceptional irrational frequencies, one can obtain local bifurcation results using the implicit function theorem as handled in \cite{Ba,Ra}. However, in the general case of $\delta = 0$ and $\nu$ irrational, one encounters the so-called small divisor problem, where periodic solutions exist only inside a Cantor-type set of Diophantine frequencies $\nu$. Existence proofs for periodic solutions to \eqref{eq:system} in this last case require an extremely delicate Nash-Moser procedure similar to those implemented in KAM theory for the proof of quasiperiodic solutions (cf. \cite{Be07,Bo95,CrWa93,GaCr15}).
\vs
In the following section, after reformulating our problem in a suitable Sobolev space $\mathscr H$, we prove that the conditions $\nu \in \bq$ and $\delta > 0$ are sufficient to guarantee the compactness of $\mathscr L^{-1}$. Since the growth conditions \ref{a3} guarantee 
the continuity of the relevant superposition operator (cf. \cite{Moshe}), we can conclude that the assumptions \ref{a1}--\ref{a3}, $\delta > 0$ and  $\nu \in \bq$ permit the application of degree theory to study the global bifurcation of \eqref{eq:system}. We emphasize two additional novelties of our approach: $(\rm i)$ the superposition operator framework accommodates a broader class of nonlinearities than is typically considered---for instance, the approach of Kosovolich et al. in \cite{Koso1, Koso2, Koso3} is limited to analytic nonlinearities, allowing for an argument based on Banach algebras which is not employed in this paper--- $(\rm ii)$ since the variable $\nu \in \bq$ is fixed and our bifurcation parameters $(\alpha,\beta) \in \br \times \br$ are unrelated to the frequency of our model, the bifurcation results obtained in this paper are not Hopf bifurcations but rather represent a distinct bifurcation phenomenon. 
\vs
The remainder of the paper is organized as follows:
In Section \ref{sec:local_global_bif}, we recall equivariant analgoues of the classical Krasnosel'skii and Rabinowitz theorems, and apply  equivariant degree theory methods to establish local and global bifurcation results for \eqref{eq:system}. While laying the groundwork for these results, we describe a symmetric decomposition of our functional space $\mathscr H$, collect spectral data relating to the linearization $D_u \mathscr F(\alpha,\beta,0): \mathscr H \rightarrow \mathscr H$ and consider a particular fixed point reduction of our problem which enables us to distinguish between stationary and non-stationary solutions.
\vs
In order to demonstrate the versatility of our method, we dedicate Section \ref{sec:n_strings} to an analysis of the dynamical system described by \eqref{eq:system} with the coupling matrix
\begin{align*}
A(\alpha) := \zeta(\alpha)(L + \id):V \rightarrow V,
\end{align*}
where $\zeta: \br \rightarrow \br$ is any continuous, strictly monotonic function and $L:V \rightarrow V$ is the graph Laplacian of an undirected graph with $N$ vertices that are invariant under the permutation action of $\Gamma$. Without specifying the group $\Gamma$, we can always assume that a complete list $\{ \mathcal V_j \}_{j=0}^r$ of the irreducible $\Gamma$-representations is made available.
Corresponding to every irreducible $\Gamma$-representation $\mathcal V_j$, the
$\Gamma$-equivariant, symmetric positive-semidefinite matrix $L$ has some number $m_j$ of real, non-negative eigenvalues $z_{j,k}$, each associated with an eigenspace $E_{j,k} \simeq \mathcal V_j$, such that one has
\[
V = \bigoplus_{j=0}^r \bigoplus_{k=1}^{m_j} \mathcal V_j, \quad A(\alpha) = \bigoplus_{j=0}^r \bigoplus_{k=1}^{m_j} \zeta(\alpha)(z_{j,k}+1) \id_{\mathcal V_j}.
\]
We observe $\rm (i)$ that our system \eqref{eq:system} admits the symmetry group
\[
G:= S^1 \times \bz_2 \times \bz_2 \times \Gamma,
\]
$\rm(ii)$ that our operator $\mathscr F: \mathscr H \rightarrow \mathscr H$ is a $G$-equivariant compact perturbation of identity with respect to the isometric $G$-action
\[
(e^{i \theta},\kappa_1,\kappa_2,\gamma)u(t,x) := \kappa_1 \gamma u(t + \theta, \kappa_2 x), \quad u \in \mathscr H,
\]
where $(\theta,\kappa_1,\kappa_2,\gamma) \in S^1 \times \bz_2 \times \bz_2 \times \Gamma$ and $\rm (iii)$ that our functional space $\mathscr H$ admits a $G$-isotypic decomposition modeled on the irreducible $G$-representations
\[
\mathcal W_m \otimes \mathcal V_j^0 \text{ and } \mathcal W_m \otimes \mathcal V_j^1, \quad m \in \bn \cup \{0\}, \; j \in \{0,1,\ldots,r \},
\]
where the notations $\mathcal V_j^0$ and $\mathcal V_j^1$ indicate that the $j$-th irreducible $\Gamma$-representation $\mathcal V_j$ has been equipped with the corresponding $\bz_2 \times \bz_2 \times \Gamma$-action
\begin{align*}
(\kappa_1,\kappa_2,\gamma)u :=    \begin{cases}
        \kappa_1 \gamma u & \text{ if } u \in \mathcal V_j^0; \\
        \kappa_1 \kappa_2 \gamma u & \text{ if } u \in \mathcal V_j^1,
    \end{cases}
\end{align*}
and where, for each $m \in \bn$, we use $\mathcal W_m \simeq \bc$ to denote the irreducible $S^1$-representation equipped with the action
\[
e^{i \theta} w := e^{im \theta}w, \quad \theta \in S^1, w \in \mathcal W_m,
\]
and $\mathcal W_0 \simeq \br$ to denote the irreducible $S^1$-representation on which $S^1$ acts trivially. 
\vs
At this point, our main equivariant bifurcation results, Theorems \ref{thm:main_local_bif} and \ref{thm:main_global_bif}, can be applied to demonstrate the existence of unbounded branches of non-stationary solutions to \eqref{eq:system} exhibiting symmetries from all possible maximal orbit types emerging from an infinite number of critical points $(\alpha_0,\beta_0,0) \in \br \times \br \times \mathscr H$.
\begin{theorem} \rm \label{thm:intro_arbitrary_Gamma}
    The trivial solution to \eqref{eq:system} with the assignments $A(\alpha):= \zeta(\alpha)(L+\id)$, $\zeta:\br \rightarrow \br$ differentiable, strictly monotonic and under the conditions $\nu \in \bq$, $\delta > 0$, $\tau \neq \pi \bq$, undergoes a global bifurcation at the critical parameter values (if they exist):
\[
\alpha_{m,n,j,k} := \zeta^{-1}\left(\frac{\nu^2 m^2 - n^2 - \delta m \cot(m \tau)}{z_{j,k}+1} \right), \quad \beta_{m,n,j,k} : = \frac{\delta m}{\sin(m \tau)},
\]
for every index quadruple $(m,n,j,k) \in 2 \bn - 1 \times \bn \times \{0,\ldots,r\} \times \{1,\ldots, m_j\}$. In particular, for each maximal orbit type $(H)$ in the isotropy lattice of the irreducible $S^1 \times \bz_2 \times \bz_2 \times \Gamma$-representation $\mathcal W_m \otimes \mathcal V_j^0$, if $n$ is even, and $\mathcal W_m \otimes \mathcal V_j^1$, if $n$ is odd, the trivial solution $(\alpha_{m,n,j,k},\beta_{m,n,j,k},0) \in \br \times \br \times \mathscr H$ is a branching point for an unbounded branch of non-stationary periodic solutions $(\alpha,\beta,u) \in \br \times \br \times \mathscr H$ with symmetries at least $(H)$.
\end{theorem}
\begin{remark} \rm
Theorem \ref{thm:intro_arbitrary_Gamma} imposes no restrictions on the resonance or multiplicity of critical points and therefore does not rely on restrictive kernel dimension conditions at critical parameter values which typically accompany non-topological approaches to bifurcation problems. Notice also that every branch of non-trivial solutions to \eqref{eq:system} detected by our method is unbounded and consists only of non-stationary functions (this last fact, on account of a fixed point reduction of our bifurcation problem to a space of anti-periodic solutions and suggested by our parity limitation on the index $m$).
\end{remark}
Finally, in Section \ref{sec:example_2}, by considering the symmetry implications of Theorem \ref{thm:intro_arbitrary_Gamma} for the special case of $\Gamma = D_N$, we identify three distinct types of maximal elements in the isotropy lattices $\mathcal W_m \otimes \mathcal V_j^0$ and $\mathcal W_m \otimes \mathcal V_j^1$. Full generator descriptions of the maximal orbit types for all $(m,j) \in 2 \bn -1 \times \{0,\ldots,r\}$ are provided and we consider motivating examples of functions exhibiting each of the corresponding symmetry types (cf. Figure \ref{fig:eigenfunctions}). 
\vs
For convenience, the Appendices include an explanation of notations used, and brief introductions to the equivariant Brouwer degree, the $S^1$-equivariant degree, and the twisted equivariant degree theories. Readers who are interested in the natural generalizations of these (local) degree theories to their infinite-dimensional counterparts are referred to \cite{book-new,AED}.
\section{Functional Space Reformulation of \eqref{eq:system}} \label{sec:reformulation}
In this section, the bifurcation problem \eqref{eq:system} is prepared for application of the twisted equivariant degree by a two-parameter operator equation reformulation in a suitable functional space. 
\vs
Define the set $\Om:= S^1\times (-\frac{\pi}{2},\frac{\pi}{2})$, where we have made the identification $S^1 \simeq \br/2\pi \bz$, and consider the Sobolev space 
\begin{align}\label{def:space_H}
\mathscr H := H^2(\Om; V) \cap H^1_0(\Om; V),
\end{align}
equipped with the norm
\begin{align} \label{def:sobolev_norm}
    \|u\|_{2}:=\max\{\|D^s u\|_{L^2}: |s|\le 2\},
\end{align}
where $s := (s_1,s_2)$, $|s|:=s_1+s_2 \leq 2$, and $D^s u :=\frac{\partial ^{|s|} \varphi}{\partial^{s_1}t\partial^{s_2}x}$.
\vs
In what follows, $V^c$ is used to indicate the complexification of the orthogonal $\Gamma$-representation $V$ (identified, in the usual way, with the tensor product $\bc \otimes V$) and we denote by $\mathscr H^c:=  H^2(\Om; V^c) \cap H^1_0(\Om; V^c)$ the corresponding complexification of the space \eqref{def:space_H}. 
\vs
Every function $u\in \mathscr H^c$ admits a complex Fourier expansion of the form
\begin{equation}\label{eq:Fourier}
u(t,x) = \sum\limits_{(m,n) \in \bz \times \bn} e^{imt}v_n(x) c_{m,n}, \quad c_{m,n} \in V^c,
\end{equation}
where 
\[
v_n(x) :=
\begin{cases}
    \cos(nx), \quad & \text{ if } n \text{ is odd}; \\
    \sin(nx), \quad & \text{ otherwise.}
\end{cases}
\]
On the other hand, a function of the form \eqref{eq:Fourier} belongs to $\mathscr H$ if and only if the following conditions are satisfied for all $(m,n) \in \bn \times \bn$:
\begin{align*}
%\label{condition:fourier_coefficients}
    \begin{cases}
         c_{-m,n} = \overline{c_{m,n}}; \\
         c_{0,n} \in V.
    \end{cases}
\end{align*}
Therefore, the standard Sobolev norm \eqref{def:sobolev_norm} is equivalent to the norm $\| \cdot \|_{\mathscr H}: \mathscr H \rightarrow \br$ defined with respect to the Fourier expansion \eqref{eq:Fourier} of any element in $\mathscr H$, as follows
\begin{align}\label{def:norm_H}
  \| u \|^2_{\mathscr H}:= \sum_{\substack{m \geq 0\\ n \in \bn}} | c_{m,n} |^2 (m^2 + n^2 + 1)^2.
\end{align}
We also consider the closed, densely-defined, differential operator 
\begin{align}\label{def:operator_L}
\mathscr L: \mathscr H \rightarrow L^2(\Om; V), \quad \mathscr L := \nu^2 \partial_t^2 - \partial_x^2 + \delta \partial_t + \id,
\end{align}
which has the complex eigenvalues 
\begin{align}\label{def:eigenvalues_L}
\xi_{m,n} := -\nu^2 m^2 + n^2 + i\delta m + 1, \quad (m,n) \in \bn \cup \{0\} \times \bn,
\end{align}
with the corresponding
complex eigenspaces
\begin{align}\label{def:eigenspace_L_complex}
\mathscr E^c_{m,n} := \{ e^{imt}v_n(x) c: c \in V^c \} \subset \mathscr H^c.
\end{align}
\begin{remark} \rm \label{rm:eigenvalues_L}
Notice that the eigenvalues $\xi_{m,n} \in \sigma (\mathscr L)$ satisfy
\begin{align} \label{condition:eigenvalues_L}
    \begin{cases}
        \xi_{m,n} \neq 0, \quad (m,n) \in \bn \cup \{0\} \times \bn; \\
        \overline{\xi_{m,n}}=\xi_{-m,n}.
    \end{cases}
\end{align}
\end{remark}
\begin{lemma} \rm \label{lemm:eigenvalues_L_estimate}
    If $\nu \in \bq$ and $\delta > 0$, then there exists $C > 0$ such that, for all $m \geq 0$ and $n \in \bn$, one has
    \begin{align} \label{rel:estimate_global_mn}
            | \xi_{m,n} | \geq C( n + m ).
    \end{align}
\end{lemma}
\begin{proof}
Without loss of generality, suppose that $\nu = \frac{p}{q}$ for some $p,q \in \bn$ and notice that
    \[
    | \xi_{m,n} | \geq | n^2 - \frac{p^2}{q^2}m^2 | - 1 + \delta m.
    \]
Let us first show that there exist positive numbers $D,M > 0$ with
\begin{align} \label{rel:estimate_local_mn}
    | \xi_{m,n} | \geq D( n + m),
\end{align}
for all $m \geq M$.
Indeed, we can always choose $M > 0$ large enough such that
\begin{align} \label{rel:lower_bound_m}
    \delta m \geq 2, \; \forall_{m \geq M},
\end{align}
Then, in the case that $n = \frac{p}{q}m$, one has
   \begin{align*}
           && | \xi_{m,n} | &\geq \delta  m - 1 &&   \\
           && & \geq \frac{\delta}{2} m = \frac{\delta}{4}(\frac{q}{p}n + m) && \text{(from } \frac{\delta}{2}m -1 > 0 \text{)} \\
           && & \geq D' (n + m), &&
   \end{align*}
and \eqref{rel:estimate_local_mn} follows with 
\[
D' := \frac{ \delta}{4} \min(\frac{q}{p}, 1).
\]
Supposing instead that $n \neq \frac{p}{q} m$, one has
   \begin{align*}
           && | \xi_{m,n} | &\geq |n - \frac{p}{q} m| | n + \frac{p}{q} m| + 
 \delta  m - 1 &&   \\
           && & \geq |n - \frac{p}{q} m| | n + \frac{p}{q} m| + \frac{\delta}{2}m, && \text{(from } \frac{\delta}{2}m -1 > 0 \text{)}
   \end{align*}
which, together with
\[
   | n - \frac{p}{q} m | = | \frac{qn-pm}{q} | > \frac{1}{q},
\]
yields
   \begin{align*}
           && | \xi_{m,n} | &\geq \frac{1}{q}(n + \frac{p}{q} m) + \frac{\delta}{2}m &&  \text{(from } | n - \frac{p}{q} m | > \frac{1}{q} \text{)} \\
           && & \geq D^{''}(n + m), && 
   \end{align*}
%\begin{align*}
%   | \xi_{m,n} | &\geq \frac{1}{q}(n + \frac{p}{q} m) + \frac{\delta}{2}m \\
%   & \geq D^{''}(n + m),
%\end{align*}
where \eqref{rel:estimate_local_mn} follows with
\[
D^{''} := \frac{1}{q}\min\{1,\frac{p}{q}+\frac{\delta}{2}q\}.
\]
Next, let us show that there exists a number $N > 0$, such that for any $m \in \bn \cup \{0\}$ with $\delta m < 2$ and for all $n \geq N$ the relation \eqref{rel:estimate_local_mn} is satisfied. Indeed, suppose that $m < M$ (here $M > 0$ is understood as the smallest positive number for which \eqref{rel:lower_bound_m} holds) and choose $N > 0$ sufficiently large such that 
\[
n^2 \geq n + \nu^2 \frac{4}{\delta^2} + 1, \; \forall_{n \geq N}.
\]
Then, since 
   \begin{align*}
           &&  n^2 - \nu^2 m^2 &> n^2 - \nu^2  \frac{4}{\delta^2} &&  \text{(from } \delta m < 2 \text{)} \\
           && & \geq n+1 > 0, && \text{(from } n^2 \geq n + \nu^2 \frac{4}{\delta^2} + 1 \text{)}
   \end{align*}
and
\begin{align*}
    \nu^2\frac{4}{\delta^2} - \nu^2 m^2 > 0, &&  \text{(from } \delta m < 2 \text{)}
\end{align*}
one has
   \begin{align*}
           && | \xi_{m,n} | & > n^2 - \nu^2m^2 - 1 + \delta m &&  \text{(from } n^2 - \nu^2m^2 > 0 \text{)} \\
           && & \geq n + \nu^2 \frac{4}{\delta^2} - \nu^2m^2 + \delta m && \text{(from } n^2 \geq n + \nu^2 \frac{4}{\delta^2} + 1 \text{)} \\
           && & \geq n + \delta m, && \text{(from } \nu^2\frac{4}{\delta^2} - \nu^2 m^2 > 0 \text{)}
   \end{align*}
such that \eqref{rel:estimate_local_mn} follows with
\[
D^{'''} := \min (1,\delta).
\]
Now, from Remark \ref{rm:eigenvalues_L}, it is enough to take
\begin{align*} %\label{eq:estimate_constant_mn}
  C = \min_{\substack{0 \leq m < M\\ 0 < n < N}}\{D', D^{''}, D^{'''}, \frac{|\xi_{m,n}|}{ m+n}\},  
\end{align*}
and the relation \eqref{rel:estimate_global_mn} is satisfied for all $(m,n) \in \bn \cup \{0\} \times \bn$.
\end{proof}
\vs
\noindent Motivated by the estimate established in Lemma \ref{lemm:eigenvalues_L_estimate}, we introduce the following assumptions on the scalar-variables $\delta$ and $\nu$ from system \eqref{eq:system}.
\begin{enumerate} [label=($A_0$)] 
    \item\label{a0} $\delta > 0$ and $\nu \in \bq$.
\end{enumerate}
The following Lemma is essential to verifying the applicability a Leray-Schauder degree based approach to the bifurcation problem \eqref{eq:system}.
\begin{lemma} \rm \label{lemm:compact_L}
    Under condition \ref{a0}, the operator $\mathscr L^{-1}|_{H^1(\Om; V)}: H^1(\Om; V) \rightarrow \mathscr H$ is bounded.
\end{lemma}
\begin{proof}
Let's equip the Sobolev space $H^1(\Om; V) \subset L^2(\Om; V)$ with the norm
\begin{align*}
%\label{def:norm_H1}
  \| u \|_{H^1}^2 := \sum\limits_{\substack{m \geq 0\\ n \in \bn}} | c_{m,n} |^2 (m^2 + n^2 + 1)^1,   
\end{align*}
and show that $\mathscr L^{-1} u \in \mathscr H$ for every $u \in H^1(\Om; V)$. Indeed, rearranging the terms of relation \eqref{rel:estimate_global_mn}, one can verify that there exists a $C' > 0$ (for example, take $C' = \frac{1}{C^2}$) with
\begin{align*} 
%\label{rel:estimate_3}
     C' \geq \frac{m^2 + n^2 + 1}{| \xi_{m,n} |^2},
\end{align*}
such that, for any $u \in H^1(\Om; V)$, one has 
\begin{align*}
    \| \mathscr L^{-1}u \|_{\mathscr H}^2 &= \sum_{\substack{m \geq 0\\ n \in \bn}} \frac{| c_{m,n} |^2}{|\xi_{m,n}|^2} (m^2 + n^2 + 1)^2 \\
    & \leq C' \sum_{\substack{m \geq 0\\ n \in \bn}} 
    | c_{m,n} |^2 (m^2 + n^2 + 1) = C'
    \| u \|_{H^1}^2.
\end{align*}
\end{proof}
\vs
Choosing
\begin{align} \label{def:lebesgue_embedding_order}
 \eta > \max \{\eta_1,\eta_2\},   
\end{align}
the {\it Nemytskii operator}
\begin{align} \label{def:nemystskii_operator}
N_f: W^{1,\eta}(\Om;V) \rightarrow H^1(\Om;V), \quad N_f(u)(t,x) : = f(u(t,x)),
\end{align}
is continuous (cf. \cite{Moshe}) and (by the Rellich-Kondrachov Theorem)  the Sobolev embedding 
\begin{align*} %\label{def:sobolev_embedding}
j: \mathscr H \rightarrow W^{1,\eta}(\Om; V), \quad j(u)(t,x) := u(t,x),
\end{align*}
is compact. 
\begin{lemma} \rm \label{lemm:nemystskii_operator_frechet}
Under condition \ref{a3},
the Nemystskii operator \eqref{def:nemystskii_operator} is Fréchet differentiable at the origin with
\[
D N_f(0) \equiv 0.
\]
\end{lemma}
\begin{proof}
It will be sufficient to demonstrate that the bounded linear operator
\[
D N_f(0): W^{1,\eta}(\Om;V) \rightarrow H^1(\Om;V), \quad [D N_f(u) v](t,x) := f'(u(t,x))v(t,x),
\]
satisfies the condition
\begin{align} \label{eq:frechet_condition}
\lim\limits_{\| v \|_{W^{1,\eta}} \to 0} \frac{\| N_f(v) - N_f(0) - D N_f(0) v \|_{H^1}}{\| v \|_{W^{1,\eta}}}  = 0.  
\end{align}
Indeed, notice that assumption \ref{a3} implies 
\begin{align} \label{eq:little_o_notation_f}
    \lim\limits_{u \rightarrow 0} \frac{| f(u) |}{| u |} = 0.
\end{align}
In turn, \eqref{eq:little_o_notation_f} implies
\begin{align*}
    f(0) = f'(0) = 0,
\end{align*}
such that \eqref{eq:frechet_condition} becomes
\begin{align} 
\label{eq:frechet_condition_simplified}
     \lim\limits_{\| v \|_{W^{1,\eta}} \to 0} \frac{\| N_f(v) \|_{H^1}}{\| v \|_{W^{1,\eta}}} = 0.
\end{align}
Combining \ref{a3} and \eqref{def:lebesgue_embedding_order} with the triangle and Hölder inequalities, one obtains
\begin{align} \label{eq:estimate_f}
    \| f(v) \|_{L^2} &= \| a |v|^{\frac{\eta_0}{2}} \|_{L^2} \nonumber \\
    &=  a \left( \int_\Om |v|^{\eta_0} \right)^{\frac{1}{2}} \nonumber \\
    & \leq a 2^{\frac{\eta - \eta_0}{2\eta}}\pi^{1 - \eta_0/\eta} \left( \int_\Om |v|^{\eta_0 \frac{\eta}{\eta_0}} \right)^{\frac{1}{\eta}\frac{\eta_0}{2}} \nonumber \\
    &= a 2^{\frac{\eta - \eta_0}{2\eta}}\pi^{1 - \eta_0/\eta} \| v \|_{L^\eta}^{\frac{\eta_0}{2}}.
\end{align}
Similarly, one finds
\begin{align} \label{eq:estimate_f'}
    \| f'(v) \|_{L^2} \leq b 2^{\frac{\eta - \eta_1}{2\eta}}\pi^{1 - \eta_1/\eta} \| v \|_{L^\eta}^{\frac{\eta_1}{2}},
\end{align}
where, by assumption, we have $a,b,c < \infty$ and $v \in W^{1,\eta}(\Om;V)$. 
Using the standard chain-rule to find the pair of partial derivatives
\begin{align*}
    \partial_t N_f(v)(t,x) = f'(v(t,x)) \partial_t v(t,x), \quad \partial_x N_f(v)(t,x) = f'(v(t,x)) \partial_x v(t,x),
\end{align*}
and applying the Hölder inequality again yields
\begin{align} \label{eq:boundedness_nemystskii_operator}
    \| N_f(v) \|_{H^1}^2 &= \| N_f(v) \|_{L^2}^2 + \| f'(v) \partial_t v \|_{L^2}^2 + \| f'(v) \partial_x v \|_{L^2}^2   \\
    & \leq   \| f(v) \|_{L^2}^2 + \|  f'(v) \|_{L^2}^2 \| \partial_t v \|_{L^2}^2  + \| f'(v)  \|_{L^2}^2 \| \partial_x v \|_{L^2}^2. \nonumber
\end{align}
On the other hand, from the assumption $\eta > 2$, we are guaranteed the existence of a constant $C > 0$ such that
\begin{align}\label{eq:estimate_sobolev_embedding}
\| v \|^2_{H^1} \leq C \| v \|^2_{W^{1,\eta}}, \quad v \in W^{1,\eta}(\Om;V).    
\end{align}
Together with the identity \eqref{eq:frechet_condition_simplified}, the estimates \eqref{eq:boundedness_nemystskii_operator} and \eqref{eq:estimate_sobolev_embedding} yield the desired
condition \eqref{eq:frechet_condition} as follows
\begin{align*}
  \lim\limits_{\| v \|_{W^{1,\eta}} \to 0} \frac{\| N_f(v) \|_{H^1}^2}{\| v \|_{W^{1,\eta}}^2} 
  & \leq \lim\limits_{\| v \|_{W^{1,\eta}} \to 0} C\frac{\| f(v) \|_{L^2}^2 + \|  f'(v) \|_{L^2}^2 \| \partial_t v \|_{L^2}^2  + \| f'(v)  \|_{L^2}^2 \| \partial_x v \|_{L^2}^2}{\| v \|_{L^2}^2 + \| \partial_t v \|_{L^2}^2 + \| \partial_x v \|_{L^2}^2} \\
  & \leq \lim\limits_{\| v \|_{W^{1,\eta}} \to 0} C\frac{\| f(v) \|_{L^2}^2}{\| v \|_{L^2}^2} + 2C\|  f'(v) \|_{L^2}^2 = 0, 
\end{align*}
where
\[
\lim\limits_{\| v \|_{W^{1,\eta}} \to 0}  \frac{\| f(v) \|_{L^2}}{\| v \|_{L^2}} = 0,
\]
follows from \eqref{def:lebesgue_embedding_order}, \eqref{eq:estimate_f} and
\[
\lim\limits_{\| v \|_{W^{1,\eta}} \to 0} \|  f'(v) \|_{L^2} = 0,
\]
follows from \eqref{eq:estimate_f'}.
\end{proof}
\vs
Finally, we introduce the two parameter family of operators 
$\mathscr F: \br \times \br \times \mathscr H \rightarrow \mathscr H$
\begin{align} \label{def:operator_F}
    \mathscr F(\alpha,\beta,u) := u - \mathscr L^{-1}(N_f(j(u)) + T(\beta)j(u) - A(\alpha)j(u) ),
\end{align}
where $T: \br \rightarrow L^\Gamma(V)$ is the continuous map
\[
T(\beta)u := u - \beta S_{\tau} u. 
\]
Notice that the equation
\begin{align} \label{eq:operator_equation}
\mathscr F(\alpha,\beta,u) = 0,
\end{align}
is equivalent to the system \eqref{eq:system} in the sense that $(\alpha,\beta,u) \in \br \times \br \times \mathscr H$ satisfies \eqref{eq:system} if and only if it is a solution to \eqref{eq:operator_equation}. In what follows, \eqref{eq:operator_equation} will be called the {\it operator equation} associated with the bifurcation problem \eqref{eq:system}.
\begin{remark} \label{rm:completely_continuous_field} \rm
From the compactness of $j: \mathscr H \rightarrow W^{1,\eta}(\Om; V)$, the continuity of $N_f: W^{1,\eta}(\Om;V) \rightarrow H^1(\Om;V)$ and, under condition \ref{a0}, the boundedness (cf. Lemma \ref{lemm:compact_L}) of $\mathscr L^{-1}|_{H^1(\Om; V)}: H^1(\Om; V) \rightarrow \mathscr H$, it follows that
$\id -\mathscr F(\alpha,\beta): \mathscr H \rightarrow \mathscr H$ is compact for all parameter pairs $(\alpha,\beta) \in \br \times \br$.
\end{remark}
\section{Local and Global Bifurcation in \eqref{eq:system}} \label{sec:local_global_bif}
In this section we present a general and replicable twisted-equivariant degree approach for solving a two-parameter symmetric bifurcation problem of the form \eqref{eq:operator_equation}. 
\vs
Consider the group
\[
G:= S^1 \times \bz_2 \times \bz_2 \times \Gamma,
\]
and notice that the space $\mathscr H$ is a natural Hilbert $G$-representation with respect to the isometric $G$-action given by
\begin{align} \label{def:isometric_G_action}
    (e^{i\theta}, \kappa_1, \kappa_2, \sigma)u(t,x) := \kappa_1 \sigma u(t + \theta, \kappa_2 x), \quad (t,x) \in \Om, \; u \in \mathscr H,
\end{align}
where $(\theta, \kappa_1, \kappa_2, \sigma) \in S^1 \times \bz_2 \times \bz_2 \times \Gamma$ and in particular, $\kappa_1,\kappa_2 = \pm 1$. Moreover, under assumptions \ref{a0}--\ref{a3}, the two-parameter family of operators $\mathscr F: \br \times \br \times \mathscr H \rightarrow \mathscr H$ $\rm (i)$ is $G$-equivariant with respect to the group action \eqref{def:isometric_G_action}, $\rm (ii)$ satisfies $\mathscr F(\alpha,\beta,0) = 0$ 
for all parameter pairs $(\alpha,\beta) \in \br \times \br$ and $\rm (iii)$ is differentiable at $0 \in \mathscr H$ with
    \[
    D_u \mathscr F(\alpha, \beta,0) = \id - \mathscr L^{-1}(T(\beta)-A(\alpha)): \mathscr H \rightarrow \mathscr H.
    \]
In what follows, we will use the notation
    \begin{align} \label{def:operator_A}
        \mathscr A(\alpha,\beta)u := u - \mathscr L^{-1}(T(\beta)u-A(\alpha)u), \quad u \in \mathscr H.
    \end{align}
The set of all solutions to the operator equation \eqref{eq:operator_equation} can be separated into the set of {\it trivial solutions}
\[
M := \{ (\alpha,\beta,0) \in \br \times \br \times \mathscr H \},
\]
and the set of {\it non-trivial solutions}
\[
\mathscr S := \{ (\alpha,\beta,u) \in \br \times \br \times \mathscr H: \mathscr  F(\alpha,\beta,u) = 0, \; u \neq 0 \}.
\]
Moreover, given any orbit type $(H) \in \Phi_1^t(G)$ (cf. Appendix \ref{sec:appendix_twisted_deg}) we can always consider the $H$-fixed-point set 
\[
\mathscr S^H := \{ (\alpha,\beta,u) \in \mathscr S : G_u \leq H \},
\]
consisting of all non-trivial solutions to \eqref{eq:operator_equation} with {\it symmetries at least $(H)$}, i.e. $(\alpha,\beta,u) \in \mathscr S^H$ if and only if $\mathscr F(\alpha,\beta,u) = 0$, $u \in \mathscr H \setminus  \{0\}$ and
\[
h u(t,x) = u(t,x) \text{ for all } h \in H \text{ and } (t,x) \in \Om.
\]
\subsection{The Local Bifurcation Invariant and Krasnosel'skii's Theorem}\label{sec:local-bif-inv}
For simplicity of notation, we identify $\br \times \br$ with the complex plane $\bc$ by associating each pair of parameters $(\alpha,\beta) \in \br \times \br$ with the complex number $\lambda := \alpha + i \beta$. In order to formulate a Krasnosel'skii-type local equivariant bifurcation result for the equation \eqref{eq:operator_equation}, it will be necessary to introduce a lexicon of bifurcation terminology, following \cite{book-new}: 
\vs
First, we clarify what is meant by a bifurcation of the equation \eqref{eq:operator_equation}:
\begin{definition} \rm
    A trivial solution $(\lambda_0,0) \in M$ is said to be a {\it bifurcation point} for \eqref{eq:operator_equation} if every open neighborhood of the point $(\lambda_0,0)$ has a non-empty intersection with the set of non-trivial solutions $\mathscr S$.
\end{definition}
It is well-known that a necessary condition for a trivial solution $(\lambda,0) \in M$ to be a bifurcation point for the equation \eqref{eq:operator_equation} is that the linear operator $\mathscr A(\lambda): \mathscr H \rightarrow \mathscr H$ is {\it not} an isomorphism. This leads to the following definition:
\begin{definition} \rm
    A trivial solution $(\lambda_0,0) \in M$ is said to be a \textit{regular point} for \eqref{eq:operator_equation} if $\mathscr A(\lambda_0): \mathscr H \rightarrow \mathscr H$ is an isomorphism and a \textit{critical point} otherwise. We call the set of all critical points 
\begin{align*} %\label{def:critical_set}
    \Lambda := \{ (\lambda,0) \in \br^2 \times \mathscr H : \mathscr A(\lambda): \mathscr H \rightarrow \mathscr H \text{ is not an isomorphism}\},    
\end{align*}
the {\it critical set} for \eqref{eq:operator_equation}. A critical point $(\lambda_0,0) \in \Lambda$ is said to be \textit{isolated} if there exists an $\varepsilon > 0$ neighborhood $B_{\varepsilon}(\lambda_0) := \{ (\lambda,0) \in \bc \times \mathscr H : \vert \lambda - \lambda_0 \vert < \varepsilon \}$ with
    \[ \overline{B_{\varepsilon}(\lambda_0)}\cap \Lambda = \{ (\lambda_0,0) \}.
    \]
\end{definition}
The next two definitions concern the continuation and symmetric properties of non-trivial solutions emerging from the critical set.
\begin{definition} \rm
 An isolated critical point $(\lambda_0,0) \in \Lambda$ is said to be a {\it branching point} for \eqref{eq:operator_equation} if there exists a non-trivial continuum $K \subset \overline{ \mathscr S}$ with $K \cap M = \{ (\lambda_0,0) \}$ and the maximal connected component $\mathscr C \subset \overline{\mathscr S}$ containing $K$ is called the \textit{branch} of nontrivial solutions with branching point $(\lambda_0,0)$.   
\end{definition}
\begin{definition} \rm
 A branch of solutions $\mathscr C$ is said to have {\it symmetries at least $(H)$} if there exists a non-trivial continuum $K \subset \overline{ \mathscr S}$ with $\mathscr C \cap \mathscr S^H = K$.
\end{definition}
Now that we have attended to the necessary preliminaries, let $(\lambda_0,0) \in \Lambda$ be an isolated critical point for \eqref{eq:operator_equation} with a deleted $\varepsilon$-neighborhood 
\begin{align} \label{def:deleted_epsilon_regular_nbhd}
\{ (\lambda,0) \in \br \times \br \times \mathscr H : 0 < \vert \lambda - \lambda_0 \vert < \varepsilon \},    
\end{align}
on which $\mathscr A(\lambda) : \mathscr H \rightarrow \mathscr H$ is an isomorphism, and choose $\upsilon > 0$ sufficiently small such that
\begin{align} \label{eq:lbi_admissibility_condition}
\{ (\lambda,u) \in \br \times \br \times \mathscr H : | \lambda - \lambda_0 | = \varepsilon, \; \Vert v \Vert_{\mathscr H} \leq \upsilon \} \cap \overline{\mathscr S} = \emptyset.    
\end{align}
We call the $G$-invariant set
\begin{align} \label{def:isolating_cylinder}
 \mathscr O := \{ (\lambda,u) \in \br \times \br \times \mathscr H : | \lambda - \lambda_0 | < \varepsilon, \; \Vert u \Vert_{\mathscr H} < \upsilon \},   
\end{align}
an {\it isolating cylinder} at $(\lambda_0,0)$ and 
a $G$-invariant function $\theta: \br \times \br \times \mathscr H \rightarrow \br$ is said to be an {\it auxiliary function} on $\mathscr O$ if it satisfies
\begin{align} \label{def:auxiliary_function_conditions}
   \begin{cases}
\theta(\lambda,0) < 0 \quad & \text{ for } |\lambda - \lambda_0 | = \varepsilon; \\
\theta(\lambda,u) > 0  \quad & \text{ for } |\lambda - \lambda_0 | \leq \varepsilon \text{ and } \Vert u \Vert_{\mathscr H} = \upsilon.
\end{cases}    
\end{align}
\begin{remark} \rm \label{rm:example_auxilliary_function}
For example, we can always use the auxiliary function
\begin{align} \label{def:auxiliary_function}
    \theta(\lambda,u) := \frac{\varepsilon}{2} - |\lambda - \lambda_0 | + \frac{2\varepsilon}{\upsilon} \| u \|_{\mathscr H}.
\end{align}   
\end{remark}
Given any auxiliary function $\theta$, the {\it complemented operator}
\begin{align} \label{def:complemented_operator_F}
\mathscr F_\theta : \br \times \br \times \mathscr H \rightarrow \br \times \mathscr H, \quad   \mathscr F_\theta(\lambda,u):= (\theta(\lambda,u),\mathscr F(\lambda,u)),
\end{align}
is an $\mathscr O$-admissible $G$-map (cf. \eqref{eq:lbi_admissibility_condition}, Appendix \ref{sec:appendix_twisted_deg}). We can now define the {\it local bifurcation invariant at $\lambda_0$}, as follows
\begin{align} \label{def:local_bif_inv}
\omega_{G}(\lambda_0) := \gdeg(\mathscr F_\theta,\mathscr O),
\end{align}
where $\gdeg$ indicates the twisted $G$-equivariant degree (cf. Appendix \ref{sec:appendix_twisted_deg}). The reader is referred to \cite{book-new,AED} for proof that the above construction of the local bifurcation invariant was independent of our choice of auxiliary function $\theta$. On the other hand, the following Krasnosel'skii-type local bifurcation result is a direct consequence of the existence property for the twisted equivariant degree (cf. Appendix \ref{sec:appendix_twisted_deg}).
\begin{theorem} \rm \label{thm:abstract_local_bif}
Suppose that $(\lambda_0,0) \in \Lambda$ is an isolated critical point for \eqref{eq:operator_equation} under conditions \ref{a0}--\ref{a3}. If there is a orbit type $(H) \in \Phi_1^t(G)$ for which 
\[
\operatorname{coeff}^H(\omega_{G}) \neq 0,
\]
(cf. \eqref{def:coefficient_operator_notation} for discussion  of the notation `$\operatorname{coeff}^H$' in the context of the $\bz$-module $A_1^t(G):= \bz[\Phi_1^t(G)]$) then there exists a branch $\mathscr C$ of non-trivial solutions to \eqref{eq:operator_equation} bifurcating from $(\lambda_0,0)$ with symmetries at least $(H)$.
\end{theorem}
\subsection{The $G$-Isotypic Decomposition of Our Functional Space}\label{sec:G_isotypic_decomp}
In order to effectively make use Theorem \ref{thm:abstract_local_bif} to determine the existence of a branch of non-trivial solutions to \eqref{eq:operator_equation} bifurcating from the zero solution, we must derive a more practical formula for the computation of the local bifurcation invariant \eqref{def:local_bif_inv}. Our first step in this direction is to describe the $G$-isotypic decomposition of $\mathscr H$, followed by an analysis of the spectrum for the two-parameter family of linear operators $\mathscr A$. 
\vs
Assuming that a complete list of the irreducible $\Gamma$-representations $\{\mathcal V_j \}_{j=0}^r$ is made available, $V$ has the unique $\Gamma$-isotypic decomposition
\begin{align*} %\label{eq:Gamma_isotypic_decomp}
    V = V_0 \oplus V_1 \oplus \cdots \oplus V_r, \quad 
\end{align*}
where each $\Gamma$-isotypic component $V_j$ is modeled on the corresponding $\Gamma$-irreducible representation $\mathcal V_j$ in the sense that $V_j$ is equivalent to the direct sum of some finite number of copies of $\mathcal V_j$, i.e.
\[
V_j \simeq \mathcal V_j \oplus \cdots \oplus \mathcal V_j.
\]
The exact number of irreducible $\Gamma$-representations $\mathcal V_j$ `contained' in the $j$-th $\Gamma$-isotypic component, denoted $m_j \in \bn$, is called the {\it $\mathcal V_j$-isotypic multiplicity} of $V$ and can be calculated according to the ratio
\begin{align*}
m_j := \dim V_j / \dim \mathcal V_j.
\end{align*}
Since $A(\alpha): V \rightarrow V$ is $\Gamma$-equivariant for all $\alpha \in \br$, every $A(\alpha)$ admits the block matrix representation
\begin{align*} %\label{eq:Gamma_block_decomp}
  A(\alpha) = \bigoplus_{j=0}^r A_j(\alpha), \quad A_j(\alpha) :=A|_{V_j}(\alpha):V_j \rightarrow V_j. 
\end{align*}
To simplify our computations, we introduce an additional condition on the family of $\Gamma$-equivariant matrices $A: \br \rightarrow L^\Gamma(V)$:
\begin{enumerate} [label=($A_4$)] 
    \item\label{a4} there exists a continuous family of diagonal matrices $D: \br \rightarrow L^\Gamma(V)$ and a non-singular matrix $P \in GL^\Gamma(V)$ such that
    \[
    D(\alpha) = P^{-1}A(\alpha)P.
    \]
\end{enumerate}
Under Condition \ref{a4}, each $A_j(\alpha):V_j \rightarrow V_j$ admits $m_j$ continuous eigenvalues $\{\zeta_{j,k}:\br \rightarrow \br \}_{k=1}^{m_j}$
corresponding to the same number of eigenspaces $ E_{j,1}, E_{j,2}, \ldots, E_{j,m_j} \subset V$ where 
\[
V_j =  E_{j,1} \oplus  E_{j,2} \oplus \cdots \oplus  E_{j,m_j}, \quad E_{j,k} \simeq \mathcal V_j,
\]
and such that, for every $j \in \{0,1,\ldots, r\}$, $\alpha \in \br$, one has
\begin{align} \label{eq:Gamma_Z2Z2_block_matrix_decomp}
\renewcommand\arraystretch{1}
A_j(\alpha) &=
\begin{blockarray}{c@{\,}ccccc}
&\text {\scriptsize$E_{j,1} $}&\text {\scriptsize$E_{j,2}$} & \cdots &\;\; \text {\scriptsize$E_{j,m_j}$} &\\
\begin{block}{c@{\,}[cccc]c}
& \zeta_{j,1}(\alpha) \id_{\mathcal V_j} & 0 & \cdots & 0\smash[b]{\vphantom{\bigg|}} &\;\;\; \text {\scriptsize$E_{j,1}$ }\\
& 0& \zeta_{j,2}(\alpha) \id_{\mathcal V_j} & \cdots &0&\;\text {\scriptsize$E_{j,2}$} \\
& \vdots & \vdots & \ddots &\; \;\; \;\vdots &\vdots \\
&0 & 0& \cdots & \zeta_{j,m_j}(\alpha) \id_{\mathcal V_j}&\;\;\; \text {\scriptsize $E_{j,m_j}$} \\
\end{block}
\end{blockarray}  \nonumber \\
&= \bigoplus_{k=1}^{m_j} \zeta_{j,k}(\alpha) \id_{\mathcal V_j}. 
\end{align}
\vs
As $\Gamma$-subrepresentations of $V$, the eigenspaces $E_{j,k} \subset V$ can be made into natural $\bz_2 \times \bz_2 \times \Gamma$-representations with antipodal $\bz_2 \times \bz_2$-action. Specifically, if we denote by $\{\mathcal V_j^0 \}_{j=0}^r$ the list of irreducible $\bz_2 \times \bz_2 \times \Gamma$-representations equipped with the $\bz_2 \times \bz_2 \times \Gamma$-action
\begin{align} \label{def:z2_gamma_action}
(\kappa_1,\kappa_2,\sigma) u := \kappa_1 \sigma u, \quad u \in \mathcal V_{j}^0,
\end{align}
where $(\kappa_1,\kappa_2,\sigma) \in \bz_2 \times \bz_2 \times \Gamma$, $\kappa_1,\kappa_2 = \pm 1$ and by $\{\mathcal V_j^1 \}_{j=0}^r$ the list of irreducible $\bz_2 \times \bz_2 \times \Gamma$-representations on which both `copies' of $\bz_2$ act antipodally, i.e.
\begin{align} \label{def:z2z2_gamma_action}
(\kappa_1,\kappa_2,\sigma) u := \kappa_1 \kappa_2 \sigma u, \quad u \in \mathcal V_{j}^1,
\end{align}
then, for every  $j \in \{0,1,\ldots,r\}$ and $k \in \{1,2,\ldots, m_j \}$ we can write $E_{j,k} \simeq \mathcal V_j^0$ (resp. $E_{j,k} \simeq \mathcal V_j^1$) to indicate that the subspace $E_{j,k} \subset V$ has been equipped with the $\bz_2 \times \bz_2 \times \Gamma$-action \eqref{def:z2_gamma_action} (resp. \eqref{def:z2z2_gamma_action}). \vs
On the other hand, for each $m \in \bn$, we denote by $\mathcal W_m \simeq \bc$ the irreducible $S^1$-representation equipped with the {\it $m$-folding $S^1$-action}
\begin{align*} %\label{def:S1_action}
    e^{i \theta}w := e^{i m\theta} w, \quad \theta \in S^1, \; w \in \mathcal W_m,
\end{align*}
and by $\mathcal W_0 \simeq \br$, the irreducible $S^1$-representation on which $S^1$ acts trivially. If we recall the complex eigenspace \eqref{def:eigenspace_L_complex} associated with each eigenvalue $\xi_{m,n} \in \bc$ of the differential operator \eqref{def:operator_L} with $m \geq 0$, $n \in \bn$ and equip the corresponding real $\mathscr L$-invariant subspace
\begin{align*}
%\label{def:eigenspace_L_real}
    \mathscr E_{m,n} &:= \{ v_n(x)(e^{imt}c + e^{-imt} \overline{c}) : c \in V^c \} \\
    & = \{v_n(x)(\cos(mt)a+\sin(mt)b): a,\; b\in V\} \subset \mathscr H, %\nonumber
\end{align*} 
with the diagonal $m$-folding $S^1$-action
\begin{align*} %\label{def:S1_diagonal_action}
 e^{i\theta} u(t,x) := u(t+m \theta,x), \quad (t,x) \in \Om, \; u \in \mathscr E_{n,m},
\end{align*}
one has
\[
\mathscr E_{n,m} \simeq \mathcal W_m \otimes V, \quad (m,n) \in \bn \cup \{0\} \times \bn,
\]
such that the $S^1$-isotypic decomposition of $\mathscr H$ is given by
\begin{align*} %\label{eq:S1_isotypic_decomp}
  \mathscr H := \overline{\bigoplus_{m = 0}^{\infty} \mathscr H_m}, \quad \mathscr H_m := \overline{\bigoplus_{n=1}^{\infty} \mathscr E_{m,n}},  
\end{align*}
where the closure is taken in $\mathscr H$. 
\vs
In the standard way, the $G$-isotypic components of $\mathscr H$ are given as $\Gamma \times \bz_2 \times \bz_2$ {\it refinements} of the $S^1$-isotypic components $\mathscr H_m$. Namely, for every $(m,n) \in \bn \cup \{0\} \times \bn$ and $j \in \{0,1,\ldots,r\}$ we define the $G$-invariant subspace
\begin{align*} %\label{def:E_mn^j}
\mathscr E_{m,n,j} := \bigoplus_{k=1}^{m_j} \mathscr E_{m,n,j,k}, \quad \mathscr E_{m,n,j,k} := \left\{ v_n(x)(\cos(mt) a + sin(mt) b)  : a, b \in E_{j,k} \right\}.
\end{align*}
Clearly, for each $(m,n) \in \bn \cup \{0\} \times \bn$, $j \in \{0,1,\ldots,r\}$ and $k \in \{1,2,\ldots, m_j \}$ one has
\[
\mathscr E_{m,n,j,k} \simeq \mathcal W_m \otimes \mathcal V_{j}^{[2 \nmid n]}, 
%\text{ and } \mathscr E_{m,n,j} \simeq \mathcal W_m \otimes \mathcal V_{j}^{[2 \nmid n]}.
\]
where, for convenience of notation, we have employed the Iverson brackets to associate any logical predicate $P$ with an element of the set $\{0,1\}$ according to the rule 
\begin{align*} %\label{Iverson_bracket}
    [P] := \begin{cases}
        1 \quad & \text{ if } P \text{ is true}; \\
        0 \quad & \text{ otherwise},
    \end{cases}
\end{align*}
such that the $G$-isotypic decomposition of $\mathscr H$ can be described in terms of the $G$-isotypic components
\begin{align} \label{eq:G_isotypic_component}
\mathscr H_{m,j}^{0} := \overline{\bigoplus_{n \in 2 \bn} \mathscr E_{m,n,j}}, \text{ and } \mathscr H_{m,j}^{1} := \overline{\bigoplus_{n \in 2 \bn + 1} \mathscr E_{m,n,j}},
\end{align}
as follows
\begin{align} \label{eq:G_isotypic_decomp}
  \mathscr H = \overline{\bigoplus_{j = 0}^r \bigoplus_{m = 0}^{\infty} \mathscr H_{m,j}^{0} \oplus \mathscr H_{m,j}^{1}}.  
\end{align}
To be clear, each of the $G$-isotypic components
$\mathscr H_{m,j}^{0}$ (resp. $\mathscr H_{m,j}^{1}$) is modeled on the irreducible $G$-representation $\mathcal V^0_{m,j} := \mathcal W_m \otimes \mathcal V_j^0$ (resp. $\mathcal V^1_{m,j} := \mathcal W_m \otimes \mathcal V_j^1$). Notice, in particular, that for $m = 0$ one has $\mathcal W_m \simeq \br$ such that $\mathcal V_{m,j}^0 \simeq \mathcal V_{j}^0$ (resp. $\mathcal V_{m,j}^1 \simeq \mathcal V_{j}^1$) for all $j \in \{0,1,\ldots,r\}$.
\vs
With the $G$-isotypic decomposition of $\mathscr H$ at hand, we can begin to describe the spectral properties of the operator \eqref{def:operator_A}. For example, by Schur's Lemma, we know that the $G$-equivariant linear operator $\mathscr A(\alpha, \beta): \mathscr H \rightarrow \mathscr H$ respects the $G$-isotypic decomposition \eqref{eq:G_isotypic_decomp} in the sense that $\mathscr A(\alpha,\beta) (\mathscr H_{m,j}^0) \subset \mathscr H_{m,j}^0$ and $\mathscr A(\alpha,\beta) (\mathscr H_{m,j}^1) \subset \mathscr H_{m,j}^1$ for all $m \geq 0$ and $j \in \{0,1,\ldots,r\}$.
Under Assumption \ref{a4}, one also has
\[
\mathscr A(\alpha,\beta)(\mathscr E_{m,n,j,k}) \subset \mathscr E_{m,n,j,k}, 
\]
for each $(\alpha,\beta) \in \br \times \br$, $(m,n) \in \bn \cup \{0\} \times \bn$, $j \in \{0,1,\ldots,r\}$ and $k \in \{1,2,\ldots,m_j\}$. 
\vs
Adopting the notations
\begin{align*} %\label{def:operator_Amnjk}
\begin{cases}
 \mathscr A_{m}(\alpha,\beta) := \mathscr A(\alpha,\beta)|_{\mathscr H_{m}}: \mathscr H_{m} \rightarrow \mathscr H_{m}, \\
 \mathscr A_{m,n,j,k}(\alpha,\beta) := \mathscr A(\alpha,\beta)|_{\mathscr E_{m,n,j,k}}: \mathscr E_{m,n,j,k} \rightarrow \mathscr E_{m,n,j,k}, 
\end{cases}
\end{align*}
every $\mathscr A(\alpha,\beta)$ admits the block-matrix decomposition
\begin{align*} %\label{eq:operator_block_decomposition_A}
    \mathscr A(\alpha,\beta) = \bigoplus_{m=0}^\infty \mathscr A_m(\alpha,\beta), \quad \mathscr A_m(\alpha,\beta) =
    \bigoplus_{n=1}^\infty \bigoplus_{j=0}^r \bigoplus_{k=1}^{m_j} \mathscr A_{m,n,j,k}(\alpha,\beta),
\end{align*}
where each of the matrices $\mathscr A_{m,n,j,k}(\alpha,\beta)$ is of the form
\begin{align*} %\label{eq:operator_Amnjk}
\mathscr A_{m,n,j,k}(\alpha,\beta) = \mu_{m,n,j,k}(\alpha,\beta) \id_{\mathcal V_{m,j}^{[2 \nmid n]}}: \mathscr E_{m,n,j,k} \rightarrow \mathscr E_{m,n,j,k},
\end{align*}
with a family of complex eigenvalues $\mu_{m,n,j,k}: \br \times \br \rightarrow \bc$ 
\begin{align} \label{def:eigenvalues_Anmjk}
\mu_{m,n,j,k}(\alpha,\beta) := 1 - \frac{ 1 - \beta e^{-i m \tau} -\zeta_{j,k}(\alpha)}{\xi_{m,n}} = \frac{-\nu^2m^2 + n^2 + i \delta m + \zeta_{j,k}(\alpha) + \beta e^{-i m \tau}}{\xi_{m,n}},
\end{align}
obtained by the computation
\begin{align*}
    T(\beta) e^{imt} = (1-\beta e^{-im\tau})e^{imt},
\end{align*}
together with the spectral descriptions \eqref{def:eigenvalues_L} and \eqref{eq:Gamma_Z2Z2_block_matrix_decomp} for the operator $\mathscr L$ and the matrix $A(\alpha)$, respectively.
In this way, we have described the spectrum of $\mathscr A(\alpha,\beta)$, in terms of the eigenvalues \eqref{def:eigenvalues_Anmjk}, as follows
\begin{align*}
    \sigma(\mathscr A(\alpha,\beta)) = \left\{ \bigcup_{m=0}^\infty \bigcup_{n=1}^\infty \bigcup_{j=0}^r \bigcup_{k = 1}^{m_j}
 \mu_{m,n,j,k}(\alpha,\beta) \right \}.
\end{align*}
\begin{remark} \label{rm:critical_values} \rm
A trivial solution $(\alpha_0,\beta_0,0) \in \br \times \br \times \mathscr H$ belongs to the critical set of \eqref{eq:operator_equation} if and only if $0 \in \sigma(\mathscr A(\alpha_0,\beta_0))$,
i.e. $(\alpha_0,\beta_0,0) \in \Lambda$ if and only if 
there exist $(m,n) \in \bn \cup \{0\} \times \bn$, 
$j \in \{0,1,\ldots,r\}$, and $k \in \{1,2,\ldots,m_j\}$
for which both
\[
\alpha_0 = \zeta_{j,k}^{-1}(\nu^2 m^2 - n^2 - \delta m \cot(m \tau)),
\]
and
\[
\beta_0 = \frac{\delta m}{\sin(m \tau)}.
\]
\end{remark}
\subsection{The $s$-Folding Homomorphism}
Before formulating and proving our main local equivariant bifurcation result, we must consider some properties of the basic degree and the orbit types of maximal kind in $\Phi_1^t(G)$ (cf. Appendix \ref{sec:appendix_twisted_deg}). For a more thorough exposition of these topics, our readers are referred to \cite{AED, book-new}.
\vs
Let's begin by associating to each $s \in \bn$
the {\it $s$-folding homomorphism} 
\begin{align*} %\label{def:s_folding_homomorphism_S1}
    \phi_s: S^1 \rightarrow S^1/\bz_s \simeq S^1, \quad \phi_s( \theta) := s\theta.
\end{align*}
There is a natural extension of $\phi_s$ to the Lie Group homomorphism
$\psi_s: S^1 \times \bz_2 \times \bz_2 \times \Gamma \rightarrow S^1 \times \bz_2 \times \bz_2 \times \Gamma$ given by
\begin{align*} %\label{def:s_folding_homomorphism_G}
    \psi_s( \theta,\kappa_1,\kappa_2, \sigma) := (s\theta,\kappa_1,\kappa_2, \sigma).
\end{align*}
In turn, each $\psi_s$ induces a $\bz$-module homomorphism $\Psi_s: A_1^t(G) \rightarrow A_1^t(G)$, defined on the generators $(H) \in \Phi_1^t(G)$, as follows
\begin{align*} %\label{def:s_folding_A1tG}
    \Psi_s(H) := ({}^{s}H), \quad {}^{s}H := \psi^{-1}_s(H).
\end{align*}
Notice that, for any $(m,n) \in \bn \cup \{0\} \times \bn$ and $j \in \{0,1,\ldots,r\}$, there is the following relation between twisted basic degrees
\begin{align*}
%\label{eq:s_folding_basicdeg}
    \Psi_s(\deg_{\mathcal V_{m,j}^{[2 \nmid n]}}) = \deg_{\mathcal V_{sm,j}^{[2 \nmid n]}}.
\end{align*}
\begin{remark} \rm
Since any orbit type which is maximal in $\Phi_1^t(G;\mathscr H \setminus \{0\})$ must also maximal in $\Phi_1^t(G;\mathscr H_{0} \setminus \{0\})$, any non-trivial function $u \in \mathscr H \setminus \{0\}$ with 
an isotropy $G_u \leq G$ such that $(G_u)$ is maximal in $\Phi_1^t(G;\mathscr H \setminus \{0\})$ must be \textit{stationary}. 
In order to detect branches of solutions to \eqref{eq:operator_equation} corresponding to maps which are both non-trivial and non-stationary, we must restrict our focus to orbit types which are maximal in $\Phi_1^t(G;\mathscr H_{m} \setminus \{0\})$ for some positive $m \in \mathbb{N}$. Such orbit types are said to be of {\it maximal kind}.
\end{remark} \rm \label{rm:maximal_kind_orbit_types}
We denote by $\mathfrak M_m$ the set of maximal elements in the isotropy lattice $\Phi_1^t(G; \mathscr H_{m} \setminus \{0\})$ and by $\mathfrak M_{m,j}^{0}$ (resp. $\mathfrak M_{m,j}^{1}$) the set of orbit types $\Phi_1^t(G; \mathscr H_{m,j}^0\setminus \{0\}) \cap \mathfrak M_m$ (resp. in $\Phi_1^t(G; \mathscr H_{m,j}^1\setminus \{0\}) \cap \mathfrak M_m$). Since every $(H) \in \mathfrak M_m$ is also an orbit type in either $\Phi_1^t(G; \mathscr H_{m,j}^0 \setminus \{0\})$ or $\Phi_1^t(G; \mathscr H_{m,j}^1 \setminus \{0\})$ for at least one $j \in \{0,1,\ldots, r\}$ we can write
\begin{align*}
    \mathfrak M_m = \bigcup_{j=0}^r \mathfrak M_{m,j}^0 \cup \mathfrak M_{m,j}^1.
\end{align*}
Clearly, one has the relation
\[
\Psi_s(\mathfrak M_{m,j}^{[2 \nmid n]}) = \mathfrak M_{sm,j}^{[2 \nmid n]},
\]
for all $m,n,s \in \bn$ and $j \in \{0,1,\ldots,r\}$. 
\vs
The following two Lemmas are a direct consequence of these observations.
\begin{lemma} \label{lemm:distjoint_maxorbtyp_sets}
$\mathfrak M_m \cap \mathfrak M_{m'} = \emptyset$ for all $m,m' \in \bn$ with $m \neq m'$.  
\end{lemma}
\begin{proof}
Indeed, since one has $\Psi_s(\mathfrak M_{1,j}^{[2 \nmid n]}) = \mathfrak M_{s,j}^{[2 \nmid n]}$, 
every orbit type $(H_0) \in \mathfrak M_{s,j}^{[2 \nmid n]}$ can be uniquely recovered from an orbit type in $\mathfrak M_{1,j}^{[2 \nmid n]}$ by the relation $(H):= \Psi_s^{-1}(H_0)$. Therefore, for any $(H_0) \in \mathfrak M_m \cap \mathfrak M_{m'}$, there exists an $(H) \in \mathfrak M_1$ with $\Psi_m(H) = \Psi_{m'}(H) = (H_0)$, and the result follows.
\end{proof}
\begin{lemma} \rm \label{lemm:basicdegree_maxorbtyps_coeff}
For any $m, > 0$, $n \in \bn$, $j \in \{0,1,\ldots,r\}$
and $(H) \in \mathfrak M_{m,j}^{[2 \nmid n]}$, one has
    \begin{align} \label{eq:non_triviality_mth_folding_coeff}
\operatorname{coeff}^{H}(\deg_{\mathcal V_{m,j}^{[2 \nmid n]}}) > 0.
\end{align}
\end{lemma}
\begin{proof}
Indeed, from the maximality of $(H)$, the recurrence formula for the twisted equivariant degree implies  
    \[
\operatorname{coeff}^{H}(\deg_{\mathcal V_{m,j}^{[2 \nmid n]}}) =  \frac{\dim {\mathcal V_{m,j}^{[2 \nmid n]}}^{H}}{2 |W(H)/S^1|},
    \]
where ${\mathcal V_{m,j}^{[2 \nmid n]}}^{H} \neq \{ 0 \}$ follows from the fact that,
since $(H) \in \Phi_0(G; \mathscr H_{m,j}^{[2 \nmid n]} \setminus \{0\})$, there exists a function $u \in \mathscr H_{m,j}^{[2 \nmid n]} \setminus \{0\}$ with an isotropy $G_x \leq G$ satisfying $(G_x) \geq (H)$.
\end{proof}
\subsection{Computation of the Local Bifurcation Invariant} \label{sec:computation_local_bif_inv}
As before, let $(\lambda_0,0) \in \Lambda$ be an isolated critical point for \eqref{eq:operator_equation} with a deleted $\varepsilon$-neighborhood \eqref{def:deleted_epsilon_regular_nbhd} on which $\mathscr A(\lambda):\mathscr H \rightarrow \mathscr H$ is an isomorphism and suppose that a number $\upsilon > 0$ is chosen such that
\[
\mathscr F(\lambda,v) \neq 0, \text{ for all } (\lambda,v) \in \br \times \br \times \mathscr H \text{ with } |\lambda - \lambda_0| = \epsilon \text{ and } 0 < \| v \|_{\mathscr H} \leq \upsilon.
\]
%$\mathscr F(\lambda,v) \neq 0$ for all $(\lambda,v) \in \br \times \br \times \mathscr H$ with $|\lambda - \lambda_0| = \epsilon$ and $0 < \| v \|_{2} \leq \delta$.
Then, for any auxiliary function $\theta:\br \times \br \times \mathscr H \rightarrow \br$ satisfying conditions \eqref{def:auxiliary_function_conditions} on the isolating cylinder \eqref{def:isolating_cylinder} (in particular, for the auxiliary function \eqref{def:auxiliary_function}), the complemented operator \eqref{def:complemented_operator_F} is $\mathscr O$-admissibly $G$-homotopic to the linear operator
\begin{align*}
    \hat{\mathscr F_\theta}(\lambda,u) := (\theta(\lambda,u),\mathscr A(\lambda)u). 
    %
    %\\ &= \bigoplus_{n=1}^\infty \bigoplus_{j=0}^r \bigoplus_{k=1}^{m_j} \mathscr A_{0,n,j,k}(\lambda)u \times
    %(\theta(\lambda,u), \bigoplus_{m=1}^\infty \bigoplus_{n=1}^\infty \bigoplus_{j=0}^r \bigoplus_{k=1}^{m_j} \mathscr A_{m,n,j,k}(\lambda)u).
\end{align*}    
%$\mathscr A_{m,n,j,k}(\alpha,\beta) \mathscr = \mu_{m,n,j,k}(\alpha,\beta) \times \id_{\mathcal U_{m,j}^{[2\nmid n]}}: \bc \times \mathscr E_{m,n,j,k} \rightarrow \bc \times \mathscr E_{m,n,j,k}$ is $\tilde{\mathscr O}_{m,n,j,k}:= \{(\lambda,u) \in \hat{\mathscr O}$-homotopic 
Adopting the notations
\begin{align*}
    \begin{cases}
        \tilde{\theta}(\lambda,u) := 
        \theta|_{\bc \times \bigoplus_{m=1}^\infty \mathscr H_{m}}(\lambda,u);\\
        \tilde{\mathscr A}(\lambda)u := \left(\tilde{\theta}(\lambda,u),\bigoplus_{m=1}^\infty \mathscr A_{m}(\lambda_0)\right); \\
        \tilde{\mathscr O} :=  \mathscr O \cap \bc \times \bigoplus_{m=1}^\infty \mathscr H_{m},
    \end{cases}
\end{align*}
and applying the homotopy property of the twisted $G$-equivariant degree (cf. Appendix \ref{sec:appendix_twisted_deg}),
the local bifurcation invariant \eqref{def:local_bif_inv} at the isolated critical point $(\lambda_0,0)$ becomes
\begin{align*} %\label{eq:local_bif_inv_reformulation1}
    \omega_G(\lambda_0) 
    = \gdeg \left(
    \tilde{\mathscr A}(\lambda_0) \times \mathscr A_{0}(\lambda_0), \;  \tilde{\mathscr O} 
 \times B(\mathscr H_{0})  \right). 
\end{align*}
%\[
%\tilde{\mathscr O}_{m,n,j,k} := \{(\lambda,u) \in \br \times \br \times \mathscr H: u \in \mathscr E_{m,n,j,k} \}.
%\]
We must impose an additional non-degeneracy assumption on the isolated critical point $(\lambda_0,0) = (\alpha_0,\beta_0,0) \in \Lambda$ in order to proceed with our computation of $\omega_G(\lambda_0) \in A_1^t(G)$:
\begin{enumerate} [label=($B_1$)] 
    \item\label{b1} For all $n \in \bn$, $j \in \{0,1,\ldots,r\}$ and $k \in \{1,2,\ldots,m_j \}$, one has
    \[
    \mu_{0,n,j,k}(\alpha_0,\beta_0) \neq 0.
    \]
\end{enumerate}
Under condition \ref{b1}, each of the matrices $\mathscr A_{0,n,j,k}(\lambda_0): \mathscr E_{0,n,j,k} \rightarrow \mathscr E_{0,n,j,k}$ is a $\bz_2 \times \bz_2 \times \Gamma$-equivariant linear isomorphism. Hence,  $(\mathscr A_0, B(\mathscr H_0))$ constitutes 
an admissible $\bz_2 \times \bz_2 \times \Gamma$-pair and, by the product property of the twisted $G$-equivariant degree (cf. Appendix \ref{sec:appendix_twisted_deg}), the local bifurcation invariant can be expressed as the $A(\bz_2 \times \bz_2 \times \Gamma)$-module product
\begin{align*} 
%\label{eq:local_bif_inv_reformulation2}
    \omega_G(\lambda_0) 
    = \Gammadeg \left( \mathscr A_{0}(\lambda_0) ,  B(\mathscr H_0) \right) \cdot
\gdeg\left( \tilde{\mathscr A}, \tilde{\mathscr O} \right),
\end{align*}
In turn, applying the product property of the $\bz_2 \times \bz_2 \times \Gamma$-equivariant degree (cf. \ref{sec:appendix_eqdeg})
\begin{align*}
%\label{eq:local_bif_inv_Gamma_product_part}
\Gammadeg &\left( \mathscr A_{0}(\lambda_0) , B(\mathscr H_{0}) \right) =\Gammadeg \left( \bigoplus_{n=1}^\infty \bigoplus_{j=0}^r \bigoplus_{k=1}^{m_j}  \mathscr A_{0,n,j,k}(\lambda_0) ,  \bigoplus_{n=1}^\infty \bigoplus_{j=0}^r \bigoplus_{k=1}^{m_j} B(\mathscr E_{0,n,j,k}) \right) \\ &= \prod_{n = 1}^\infty \prod_{j=0}^r \prod_{k=1}^{m_j} \Gammadeg(\mathscr A_{0,n,j,k}(\lambda_0) ,  B(\mathscr E_{0,n,j,k})), \nonumber
\end{align*}
and also the Splitting Lemma of the twisted $G$-equivariant degree (cf. Appendix \ref{sec:appendix_twisted_comp_form})
\begin{align*}
%\label{eq:local_bif_inv_G_sum_part}
\gdeg \left( \tilde{\mathscr A}, \tilde{\mathscr O} \right) &= \gdeg \left( \bigoplus_{m=1}^\infty \bigoplus_{n=1}^\infty \bigoplus_{j=0}^r \bigoplus_{k=1}^{m_j} \tilde{\mathscr A}_{m,n,j,k}(\lambda_0), \bigoplus_{m=1}^\infty \bigoplus_{n=1}^\infty \bigoplus_{j=0}^r \bigoplus_{k=1}^{m_j} \tilde{\mathscr O}_{m,n,j,k}  \right) \\
&= \sum_{m=1}^{\infty} \sum_{n = 1}^\infty \sum_{j=0}^r \sum_{k=1}^{m_j} \gdeg(\tilde{\mathscr A}_{m,n,j,k}(\lambda_0) ,  \tilde{\mathscr O}_{m,n,j,k}), 
\end{align*}
where we have made use of the notation
\begin{align*}
   \begin{cases}
        \theta_{m,n,j,k}(\lambda,u) := \theta|_{\bc \times \mathscr E_{m,n,j,k}}(\lambda,u); \\
        \tilde{\mathscr A}_{m,n,j,k}(\lambda)u := (\theta_{m,n,j,k}(\lambda,u),\mathscr A_{m,n,j,k}(\lambda)u); \\
       \tilde{\mathscr O}_{m,n,j,k} :=\mathscr O \cap \bc \times \mathscr E_{m,n,j,k},
    \end{cases}
\end{align*}
one obtains
\begin{align}
\label{eq:local_bif_inv_reformulation3}
    \omega_G(\lambda_0) 
    = \prod_{n = 1}^\infty \prod_{j=0}^r \prod_{k=1}^{m_j} & \Gammadeg(\mathscr A_{0,n,j,k}(\lambda_0) ,  B(\mathscr E_{0,n,j,k})) \\ & \cdot
    \sum_{m=1}^{\infty} \sum_{n = 1}^\infty \sum_{j=0}^r \sum_{k=1}^{m_j} \gdeg(\tilde{\mathscr A}_{m,n,j,k}(\lambda_0) ,  \tilde{\mathscr O}_{m,n,j,k}). \nonumber
\end{align}
\begin{lemma} \label{lemm:index_sets} \rm
    Under the assumptions \ref{a0}--\ref{a4} and for an isolated critical point $(\lambda_0,0) = (\alpha_0,\beta_0,0)$ satisfying assumption
    \ref{b1}, the formulation \eqref{eq:local_bif_inv_reformulation3} is well-defined. 
\end{lemma}
\begin{proof}
Let's prove that the expression \eqref{eq:local_bif_inv_reformulation3} involves only a finite number of non-trivial terms. Indeed, since $\id - \mathscr A(\alpha,\beta): \mathscr H \rightarrow \mathscr H$ is compact for all $(\alpha,\beta) \in \br \times \br$ (cf. Remark \ref{rm:completely_continuous_field}), the eigenvalues $\mu_{0,n,j,k}(\alpha_0,\beta_0)$ are positive (with only finitely many possible exceptions). Therefore, one has
    \[
\Gammadeg(\mathscr A_{0,n,j,k}(\alpha_0,\beta_0), B(\mathscr E_{0,n,j,k})) = (\bz_2 \times \bz_2 \times \Gamma) \in A(\bz_2 \times \bz_2 \times \Gamma),
\]
for almost every triple $(n,j,k)$ with $n \in \bn$, $j \in \{0,1,\ldots,r\}$ and $k \in \{1,2,\ldots, m_j\}$, since the $\bz_2 \times \bz_2 \times \Gamma$-equivariant degree of each $\bz_2 \times \bz_2 \times \Gamma$-equivariant linear isomorphism $\mathscr A_{0,n,j,k}(\alpha_0,\beta_0)$ on the unit ball $B(\mathscr E_{0,n,j,k})$ is fully specified by its real spectra $\sigma(\mathscr A_{0,n,j,k}(\alpha_0,\beta_0))$ and the irreducible $\bz_2 \times \bz_2 \times \Gamma$-representation $\mathcal V_j^{[2 \nmid n]}$ according to the formula
\begin{align*} %\label{eq:Gamma_deg_A_0njk}
\Gammadeg(\mathscr A_{0,n,j,k}(\alpha_0,\beta_0), B(\mathscr E_{0,n,j,k})) =    \begin{cases}
        \deg_{\mathcal V_{j}^{[2 \nmid n]}} \quad & \text{ if } \mu_{0,n,j,k}(\alpha_0,\beta_0) < 0; \\
        (\bz_2 \times \bz_2 \times \Gamma) & \text{ otherwise.}
    \end{cases}
\end{align*}
Similarly, with only finitely many possible exceptions, the matrices $\mathscr A_{m,n,j,k}(\alpha,\beta)$ are invertible. Therefore, one has
\[
\gdeg( \tilde{\mathscr A}_{m,n,j,k}(\lambda_0), \tilde{\mathscr O}_{m,n,j,k}) = (G) \in A_t^1(G),
\]
for almost every quadruple $(m,n,j,k)$ with $(m,n) \in \bn \times \bn$, $j \in \{0,1,\ldots,r\}$ and $k \in \{1,2,\ldots, m_j\}$, since the $G$-equivariant twisted degree of each complemented operator $(\theta_{m,n,j,k}, \mathscr A_{m,n,j,k})$ on its corresponding isolating neighborhood $\tilde{\mathscr O}_{m,n,j,k}$ is fully specified by the spectrum of $\mathscr A_{m,n,j,k}$ and the irreducible $G$-representation $\mathcal V_{m,j}^{[2 \nmid n]}$ according to the formula
\begin{align*} %\label{eq:twisted_deg_A_mnjk}
\gdeg( \tilde{\mathscr A}_{m,n,j,k}(\alpha_0,\beta_0), \tilde{\mathscr O}_{m,n,j,k}) = 
\begin{cases}
    \rho_{m,n,j,k}(\alpha_0,\beta_0)  \deg_{\mathcal V_{m,j}^{[2 \nmid n]}} \quad & \text{ if } \mu_{m,n,j,k}(\alpha_0,\beta_0) = 0; \\
        (G) & \text{ otherwise,}
    \end{cases}
\end{align*}    
where $\rho_{m,n,j,k}(\alpha,\beta) := \deg( \det\nolimits_\bc \mathscr A_{m,n,j,k}(\alpha,\beta))$.
\end{proof}
\vs
With motivation from Lemma \ref{lemm:index_sets} and as a matter of convenience, we introduce some notation to keep track of the indices
\begin{align} \label{def:index_set_full}
    \Sigma := \{ (m,n,j,k) : m \in \bn \cup \{0\}, \; n \in \bn, \; j \in \{0,1,\ldots,r\}, \; k \in \{1,2,\ldots,m_j\} \},
\end{align}
which non-trivially contribute to the $A_1^t(G)$-module product \eqref{eq:local_bif_inv_reformulation3}. Specifically, we define the index sets 
\begin{align} \label{def:index_set_null}
    \Sigma_0(\alpha,\beta) := \{ (m,n,j,k) \in \Sigma :  m > 0, \; \mu_{m,n,j,k}(\alpha,\beta) = 0\},
\end{align}
and
\begin{align} \label{def:index_set_negative}
    \Sigma_-(\alpha,\beta) := \{ (n,j,k): (0,n,j,k) \in \Sigma \text{ and }  \mu_{0,n,j,k}(\alpha,\beta) < 0 \},
\end{align}
to account for the {\it null} and {\it negative} spectra of $\mathscr A(\alpha,\beta): \mathscr H \rightarrow \mathscr H$, respectively. 
\begin{lemma} \rm \label{lemm:local_bif_inv_comp_formula}
    Under the assumptions \ref{a0}--\ref{a4} and using the notations \eqref{def:index_set_null}--\eqref{def:index_set_negative}, the local bifurcation invariant at an isolated critical point $(\lambda_0,0) = (\alpha_0,\beta_0,0) \in \Lambda$ satisfying assumption \ref{b1} becomes
\begin{align}\label{eq:local_bif_inv_comp_formula}
        \omega_G(\lambda_0) = \prod\limits_{(n,j,k) \in \Sigma_{-}(\alpha_0,\beta_0)} \deg_{\mathcal V_{j}^{[2 \nmid n]}} \cdot \sum\limits_{(m,n,j,k) \in \Sigma_0(\alpha_0,\beta_0)} \rho_{m,n,j,k}(\alpha_0,\beta_0) \deg_{\mathcal V_{m,j}^{[2 \nmid n]}},
    \end{align}
where each $\rho_{m,n,j,k}: \bc \rightarrow \bz$ is defined as the winding number of the corresponding complex eigenvalue $\mu_{m,n,j,k}: \bc \rightarrow \bc$, i.e.:
\begin{align}\label{def:local_bif_inv_coeff}
    \rho_{m,n,j,k}(\alpha,\beta) 
    :&= \deg( \det\nolimits_\bc \mathscr A_{m,n,j,k}(\alpha,\beta)) \\ &= \deg(\mu_{m,n,j,k}(\alpha,\beta), B_{\varepsilon}(\lambda)). \nonumber
\end{align}
\end{lemma} 
Lemmas \eqref{lemm:distjoint_maxorbtyp_sets} and \eqref{lemm:basicdegree_maxorbtyps_coeff} suggest a natural refinement of our index set \eqref{def:index_set_null}. Specifically, for a given parameter pair $(\alpha,\beta) \in \br \times \br$ and positive folding $s > 0$, we put
\begin{align} \label{def:index_set_m_folding_local}
\Sigma_s(\alpha,\beta,H) := \{ (n,j,k): (s,n,j,k) \in \Sigma_0(\alpha,\beta) \}. 
\end{align}
Clearly, a necessary condition for an orbit type of maximal kind $(H) \in \mathfrak M_s$ to satisfy the condition \eqref{eq:non_triviality_mth_folding_coeff} for a given quadruple of indices $(m,n,j,k) \in \Sigma_0(\alpha,\beta)$ is the inclusion $(n,j,k) \in \Sigma_s(\alpha,\beta)$.
\vs
We are now in a position to formulate our main local equivariant bifurcation result.
\begin{theorem} \rm \label{thm:main_local_bif}
    Under the conditions \ref{a0}--\ref{a4}, with an isolated critical point $(\lambda_0,0) = (\alpha_0,\beta_0,0) \in \Lambda$ satisfying \ref{b1} and using the notation \eqref{def:index_set_m_folding_local}, if there is a folding $s \in \bn$ and an orbit type of maximal kind $(H) \in \mathfrak M_1$ satisfying 
\begin{align*} 
    \sum_{(n,j,k) \in \Sigma_s(\alpha_0,\beta_0) } \rho_{s,n,j,k}(\alpha_0,\beta_0) \operatorname{coeff}^{{}^{s}H}(\deg_{\mathcal V_{s,j}^{[2 \nmid n]}})  \neq 0,
\end{align*}
then there exists a branch $\mathcal C$ of non-trivial solutions to problem \eqref{eq:operator_equation} bifurcating from $(\alpha_0,\beta_0,0)$ with symmetries at least $({}^{s}H)$.
\end{theorem}
\begin{proof}
The recurrence formula  for the $A(\bz_2 \times \bz_2 \times \Gamma)$-module product \eqref{def:recurrence_formula_module_product} implies 
\begin{align*}
\operatorname{coeff}^{{}^{s}H}((K) \cdot ({}^{m}H)) = 
\begin{cases}
      1 & \text{ if } (K) = (\bz_2 \times \bz_2 \times \Gamma) \text{ and } m = s; \\
      0 \quad & \text{ otherwise,}
\end{cases}
\end{align*}
for any orbit type of maximal kind $(H) \in \mathfrak M_1$ and $s \in \bn$. Moreover, since every basic degree $\deg_{\mathcal V_j^{[2 \nmid n]}} \in A(\bz_2 \times \bz_2 \times \Gamma)$ is of the form
\[
(\bz_2 \times \bz_2 \times \Gamma) - \bm \alpha_{n,j},
\]
for some $\bm \alpha_{n,j} \in A(\bz_2 \times \bz_2 \times \Gamma)$ with $\operatorname{coeff}^{\bz_2 \times \bz_2 \times \Gamma}(\bm \alpha_{n,j}) = 0$, one has
\[
\operatorname{coeff}^{{}^{s}H}(\deg_{\mathcal V_j^{[2 \nmid n]}} \cdot \deg_{\mathcal V_{m',j'}^{[2 \nmid n']}}) = \operatorname{coeff}^{{}^{s}H}(\deg_{\mathcal V_{m',j'}^{[2 \nmid n']}}), \quad m'>0, \; n,n' \in \bn, \; j,j' \in \{0,1,\ldots, r\},
\]
such that
\[
\operatorname{coeff}^{{}^{s}H}\left(\prod\limits_{(n,j,k) \in \Sigma_{-}(\alpha_0,\beta_0)} \deg_{\mathcal V_{j}^{[2 \nmid n]}} \cdot \deg_{\mathcal V_{m',j'}^{[2 \nmid n']}} \right) = \operatorname{coeff}^{{}^{s}H}(\deg_{\mathcal V_{m',j'}^{[2 \nmid n']}}).
\]
Therefore, the coefficient of $({}^{s}H)$ in the local bifurcation invariant is given by
\begin{align*} %\label{eq:thm_lbi_coeffH}
\operatorname{coeff}^{{}^{s}H}(\omega_G(\lambda_0)) &= \operatorname{coeff}^{{}^{s}H}\left( \prod\limits_{(n,j,k) \in \Sigma_{-}(\alpha_0,\beta_0)} \deg_{\mathcal V_{j}^{[2 \nmid n]}} \cdot \sum\limits_{(m,n,j,k) \in \Sigma_0(\alpha_0,\beta_0)} \rho_{m,n,j,k}(\alpha_0,\beta_0) \deg_{\mathcal V_{m,j}^{[2 \nmid n]}} \right) \nonumber \\ 
    &= \operatorname{coeff}^{{}^{s}H}\left( \sum\limits_{(m,n,j,k) \in \Sigma_0(\alpha_0,\beta_0)} \rho_{m,n,j,k}(\alpha_0,\beta_0) \deg_{\mathcal V_{m,j}^{[2 \nmid n]}} \right) \nonumber \\
    &= \sum\limits_{(m,n,j,k) \in \Sigma_0(\alpha_0,\beta_0)} \rho_{m,n,j,k}(\alpha_0,\beta_0) \operatorname{coeff}^{{}^{s}H}( 
 \deg_{\mathcal V_{m,j}^{[2 \nmid n]}} ) \nonumber \\
&= \sum\limits_{(n,j,k) \in \Sigma_s(\alpha_0,\beta_0)} \rho_{s,n,j,k}(\alpha_0,\beta_0) \operatorname{coeff}^{{}^{s}H}(\deg_{\mathcal V_{s,j}^{[2 \nmid n]}}),
\end{align*}
and the conclusion follows from Theorem \ref{thm:abstract_local_bif}.
\end{proof}
\vs
\begin{remark} \label{rm:crossing_sign} \rm
Let $(\alpha_0,\beta_0,0) \in \Lambda$ be an isolated critical point and recall that every quadruple $(m,n,j,k) \in \Sigma_0(\alpha_0,\beta_0)$ is associated with an eigenvalue $\mu_{m,n,j,k}: \bc \rightarrow \bc$ satisfying 
\[
\mu_{m,n,j,k}(\alpha_0,\beta_0) = 0.
\]
Given a positive folding $s > 0$, if the set of eigenvalues $\mu_{s,n,j,k}$ associated with triples $(n,j,k) \in \Sigma_s(\alpha_0,\beta_0)$ are all crossing the imaginary axis {\it in the same direction}, then it must be that the corresponding set of coefficients $\rho_{s,n,j,k}(\alpha_0,\beta_0) \in \bz$ all {\it share the same sign} (c.f. \eqref{def:local_bif_inv_coeff}).
\end{remark}
Remark \ref{rm:crossing_sign} provides insight into why Theorem \ref{thm:intro_arbitrary_Gamma} imposes no restrictions on the multiplicity of critical points and also suggests the following practical corollary:
\begin{corollary} \label{cor:main_local_bif}
Under the conditions \ref{a0}--\ref{a4} and with an isolated critical point $(\lambda_0,0) = (\alpha_0,\beta_0,0)$ satisfying condition \ref{b1}, if there is a positive number $s > 0$ with
    \[
    \rho_{s,n,j,k}(\alpha_0,\beta_0) \rho_{s,n',j',k'}(\alpha_0,\beta_0) \geq 0, 
    \]
for every pair of triples $(n,j,k),(n',j',k') \in \Sigma_s(\alpha_0,\beta_0)$ and an orbit type of maximal kind $(H) \in \mathfrak M_1$ satisfying $(H) \in \mathfrak M_{1,j}^{[2 \nmid n]}$ for at least one triple $(n,j,k) \in \Sigma_s(\alpha_0,\beta_0)$ with $\rho_{s,n,j,k}(\alpha_0,\beta_0) \neq 0$, then there exists a branch $\mathcal C$ of non-trivial solutions to problem \eqref{eq:operator_equation} bifurcating from $(\alpha_0,\beta_0,0)$ with symmetries at least $({}^{s}H)$.
\end{corollary}
\begin{proof}
     A direct consequence of Lemma \ref{lemm:basicdegree_maxorbtyps_coeff} and Theorem \ref{thm:main_local_bif}.
\end{proof}

\subsection{The Rabinowitz Alternative} \label{sec:rab_alt}
In order to study the global behaviour of branches of non-trivial solutions predicted by Theorem \eqref{thm:main_local_bif} with a Rabinowitz-type argument, it will be necessary to adopt the following pair of assumptions:
\begin{enumerate} [label=($\tilde{B}_1$)] 
    \item\label{b1t} For all $n \in \bn$, $j \in \{0,1,\ldots,r\}$, $k \in \{1,2,\ldots,m_j\}$ and for every critical point $(\alpha_0,\beta_0,0) \in \Lambda$, one has
    \[
    \mu_{0,n,j,k}(\alpha_0,\beta_0) \neq 0.
    \]
\end{enumerate}
\begin{enumerate} [label=($B_2$)] 
    \item\label{b2} The critical set $\Lambda \subset M$ is nonempty and discrete.
\end{enumerate}
Under conditions \ref{b1t}--\ref{b2}, the local bifurcation invariant at every critical point for the operator equation \eqref{eq:operator_equation} is well-defined and can be calculated with the computational formula \eqref{eq:local_bif_inv_comp_formula}. 
\begin{remark} \rm
As will become clear in the subsequent section, the Condition \ref{b1t} can be avoided with an appropriate fixed-point reduction.
\end{remark}
Proof of the following global equivariant bifurcation result can be found in \cite{book-new}.
\begin{theorem} \rm \label{thm:rab_alt} {\bf (Rabinowitz' Alternative)}
Suppose that conditions \ref{a0}--\ref{a4}, \ref{b1t}--\ref{b2} are satisfied and let $\mathcal U \subset \br \times \br \times \mathscr H$ be any open bounded $G$-invariant with $\partial \mathcal U \cap \Lambda = \emptyset$. If $\mathcal C \subset \mathscr S$ is a branch of non-trivial solutions to \eqref{eq:operator_equation} bifurcating from a critical point $(\lambda_0,0) \in \mathcal U \cap \Lambda$, then one has the following alternative: 
\begin{enumerate}[label=$(\alph*)$]
\item\label{alt_a}  either $\mathcal C \cap \partial \mathcal U \neq \emptyset$;
    \item\label{alt_b} or there exists a finite set
    \begin{align*}
        \overline{\mathcal C} \cap \Lambda = \{ (\lambda_0,0),(\lambda_1,0), \ldots, (\lambda_{n_0},0) \},
    \end{align*}
    satisfying the following relation
    \begin{align*}
        \sum\limits_{i=1}^{n_0} \omega_{G}(\lambda_i) = 0.
    \end{align*}
\end{enumerate} 
\end{theorem}
\begin{remark} \rm
    If a branch of non-trivial solutions $\mathcal C \subset \mathscr S$ satisfies $\mathcal C \cap \partial \mathcal U \neq \emptyset$ for every open bounded $G$-invariant set $\mathcal U \subset \br \times \br \times \mathscr H$ with $\partial \mathcal U \cap \Lambda = \emptyset$, then  $\mathcal C$ must be {\it unbounded}. Therefore, a sufficient condition for the unboundedness of a branch $\mathcal C \subset \mathscr S$ is the following:
        \begin{align*}
        \sum\limits_{(\lambda_k,0) \in\overline{\mathcal C} \cap \Lambda} \omega_{G}(\lambda_k) \neq 0.
    \end{align*}
\end{remark}

\subsection{Resolution of the Rabinowitz Alternative and the $\mathbf{H}$-Fixed Point Setting} \label{sec:res_rab_alt}
With an appropriate {\it fixed point reduction} for the problem \eqref{eq:operator_equation}, we can guarantee that a branch $\mathcal C$ of non-trivial solutions to \eqref{eq:system}, whose existence has been established by Theorem \ref{thm:main_local_bif}, is comprised only of {\it non-stationary solutions}.
\vs
Consider the subgroup
\[
\bm H := \{(1,1,1,e_\Gamma), (-1,-1,1,e_\Gamma) \} \leq S^1 \times \bz_2 \times \bz_2 \times \Gamma,
\]
and the corresponding $\bm H$-fixed point subspace $\mathscr H^{\bm H} := \{ u \in \mathscr H: h u = u \; \forall_{h \in \bm H} \}$. Adopting the notations
\begin{align} \label{def:H_fixed_operator_F}
    \mathscr F^{\bm H}: \br \times \br \times \mathscr H^{\bm H} \rightarrow \mathscr H^{\bm H}, \quad \mathscr F^{\bm H} := \mathscr F|_{\br \times \br \times \mathscr H^{\bm H}},
\end{align}
and
\begin{align} \label{def:H_fixed_operator_A}
\mathscr A^{\bm H}: \br \times \br \times \mathscr H^{\bm H} \rightarrow \mathscr H^{\bm H}, \quad \mathscr A^{\bm H} := \mathscr A|_{\br \times \br \times \mathscr H^{\bm H}},
\end{align}
notice that any solution $(\alpha,\beta,u) \in \br \times \br \times \mathscr H^{\bm H}$ to the operator equation
\begin{align} \label{eq:H_fixed_operator_equation}
    \mathscr F^{\bm H}(\alpha,\beta,u) = 0,
\end{align}
is also a solution to \eqref{eq:operator_equation}.  
\vs
In this $\bm H$-fixed point setting, we must adapt each of the notions introduced in Sections \ref{sec:computation_local_bif_inv} and \ref{sec:rab_alt} used to describe the Rabinowitz alternative for the equation \eqref{eq:operator_equation} to the {\it $\bm H$-fixed operator equation} \eqref{eq:H_fixed_operator_equation}. 
Notice first that any trivial solution to \eqref{eq:H_fixed_operator_equation} belonging to the
{\it $\bm H$-fixed critical set}
\begin{align} \label{def:H_fixed_critical_set}
    \Lambda^{\bm H} := \{ (\alpha,\beta,0) \in \br \times \br \times \mathscr H^{\bm H}: \mathscr A^{\bm H}(\alpha,\beta) \text{ is not an isomorphism} \},
\end{align}
is also a critical point of the equation \eqref{eq:operator_equation}.
Clearly, discreteness of the critical set $\Lambda$ implies discreteness of \eqref{def:H_fixed_critical_set}. We can modify the assumption \ref{b2} to ensure that $\Lambda^{\bm H}$ is also non-empty:
\begin{enumerate} [label=($\tilde{B}_2$)] 
    \item\label{b2t} the $\bm H$-fixed critical set $\Lambda^{\bm H} \subset M$ is nonempty and discrete.
\end{enumerate}
Notice also that the $\bm H$-fixed point space $\mathscr H^{\bm H}$ is an isometric Hilbert representation of the group $\bm G := W(\bm H) \simeq S^1 \times \bz_2 \times \Gamma$ and that, since the operators \eqref{def:H_fixed_operator_F}, \eqref{def:H_fixed_operator_A} are $\bm G$-equivariant, the equation \eqref{eq:H_fixed_operator_equation} can be understood as a $\bm G$-equivariant bifurcation problem. In particular, since the action of $(-1,-1) \in S^1 \times \bz_2$ on any $G$-isotypic component $\mathscr H_{m,j}^{[2 \nmid n]}$ (cf. \eqref{eq:G_isotypic_component}) is given by
\begin{align*}
    (-1,-1) v_n(x)(\cos(mt) a + \sin(mt) b) = -v_n(x)(\cos(m(t+\pi))a+ \sin(m(t+\pi))),
\end{align*}
it can easily be verified that $\mathscr H^{\bm H}$ has the $\bm G$-isotypic decomposition
\begin{align*} %\label{eq:bm_G_isotypic_decomposition}
\mathscr H^{\bm H} = \overline{\bigoplus_{j = 0}^r \bigoplus_{m \in 2 \bn - 1}^{\infty}} \mathscr H_{m,j}^{0} \oplus \mathscr H_{m,j}^{1}.
\end{align*}
Therefore, since $u \in \mathscr H$ is $S^1$-invariant (stationary) if and only if $u \in \mathscr H_{0,j}^{[2 \nmid n]}$ for some $n \in \bn$ and $j \in \{0,1,\ldots,r\}$, any solution to \eqref{eq:H_fixed_operator_equation} belonging to the set of {\it $\bm H$-fixed non-trivial solutions}
\begin{align*} %\label{def:H_fixed_nontrivial_solutions}
\mathscr S^{\bm H} := \{ (\alpha,\beta,u) \in \br \times \br \times \mathscr H^{\bm H}: \mathscr F^{\bm H}(\alpha,\beta,u) = 0 \text{ and } u \neq 0 \},
\end{align*}
must be non-stationary. Moreover, we are guaranteed that any connected component of non-trivial solutions $K \subset \mathscr S^{\bm H}$ to \eqref{eq:H_fixed_operator_equation} consists only of non-stationary solutions.
\begin{remark} \label{rm:non_stationary_branches} \rm
If $\mathcal C \subset \mathscr S$ is a branch of non-trivial solutions to \eqref{eq:operator_equation} bifurcating from an $\bm H$-fixed critical point $(\lambda_0,0) \in \Lambda^{\bm H}$, then there exists a non-empty connected component $K \subset \mathscr S^{\bm H}$ of non-stationary solutions with $\mathcal C^{\bm H} = \mathcal C \cap \mathscr S^{\bm H}$.
\end{remark}
As was the case with the spectrum of \eqref{def:operator_A}, we are able to describe the spectrum of $\mathscr A^{\bm H}(\alpha,\beta)$ for any parameter pair $(\alpha,\beta) \in \br \times \br$ in terms of the eigenvalues \eqref{def:eigenvalues_Anmjk} as follows 
\begin{align*} %\label{eq:H_fixed_spectrum_operator_A}
    \sigma(\mathscr A^{\bm H}(\alpha,\beta)) = \left\{ \bigcup_{m=1}^\infty \bigcup_{n=1}^\infty \bigcup_{j=0}^r \bigcup_{k = 1}^{m_j}
 \mu_{2m-1,n,j,k}(\alpha,\beta) \right \}.
\end{align*}
We refine the index set \eqref{def:index_set_null} to include only those indices $(m,n,j,k) \in \Sigma$ relevant to the $\bm H$-fixed point setting with the notation
\begin{align} \label{def:H_fixed_index_set}
\Sigma_0^{\bm H}(\alpha,\beta) := \{ (m,n,j,k) \in \Sigma_0(\alpha,\beta): 2 \nmid m \},
\end{align}
and, in order to distinguish between the local bifurcation invariant \eqref{def:local_bif_inv} and the {\it $\bm H$-fixed local bifurcation invariant} at a critical point $(\lambda_0,0) \in \Lambda^{\bm H}$, we can use the notation
\begin{align} \label{def:H_fixed_basic_degree}
\omega_G^{\bm H}(\lambda_0) := \bm G\text{-deg}(\mathscr F_{\theta}^{\bm H}, \mathscr O^{\bm H}),
\end{align}
where $\mathscr O^{\bm H} \subset \br \times \br \times \mathscr H^{\bm H}$ is the {\it $\bm H$-fixed isolating cylinder} at $(\lambda_0,0)$ (cf. \eqref{def:isolating_cylinder})
\begin{align*} %\label{def:H_fixed_isolating_cylinder}
 \mathscr O^{\bm H} := \{ (\lambda,u) \in \br \times \br \times \mathscr H^{\bm H} : | \lambda - \lambda_0 | < \varepsilon, \; \Vert u \Vert_{\mathscr H} < \upsilon \},   
\end{align*}
$\mathscr F_{\theta}^{\bm H}: \br \times \br \times \mathscr H^{\bm H}$ is the {\it $\bm H$-fixed complemented operator} (cf. \eqref{def:complemented_operator_F})
\begin{align*} %\label{def:H_fixed_complemented_operator_F}
    \mathscr F_{\theta}^{\bm H}(\lambda,u) := (\theta(\lambda,u), \mathscr F^{\bm H}(\lambda,u)),
\end{align*}
and $\theta: \br \times \br \times \mathscr H^{\bm H} \rightarrow \br$ is any auxilliary function on $\mathscr O^{\bm H}$ (cf. \eqref{def:auxiliary_function}). 
\vs
Lemma \eqref{lemm:local_bif_inv_comp_formula} is reformulated for the map $\mathscr F^{\bm H}$ as follows:
\begin{lemma} \rm \label{lemm:H_fixed_local_bif_inv_comp_formula}
    Under the assumptions \ref{a0}--\ref{a4}, \ref{b2} and using the notations \eqref{def:H_fixed_index_set}, \eqref{def:H_fixed_basic_degree}, the local bifurcation invariant at any $\bm H$-fixed critical point $(\lambda_0,0) = (\alpha_0,\beta_0,0) \in \Lambda^{\bm H}$ is given by
\begin{align}\label{eq:H_fixed_local_bif_inv_comp_formula}
       \omega_G^{\bm H}(\lambda_0) = \sum\limits_{(m,n,j,k) \in \Sigma^{\bm H}_0(\alpha_0,\beta_0)} \rho_{m,n,j,k}(\alpha_0,\beta_0) \deg_{\mathcal V_{m,j}^{[2 \nmid n]}},
    \end{align}
where the coefficients $\rho_{m,n,j,k}(\alpha,\beta) \in \bz$ are obtained via the formula \eqref{def:local_bif_inv_coeff}.
\end{lemma} 
Likewise, Theorem \ref{thm:rab_alt} becomes:
\begin{theorem} \rm \label{thm:H_fixed_rab_alt}{\bf ($\bm H$-Fixed Rabinowitz' Alternative)}
Suppose that conditions \ref{a0}--\ref{a4}, \ref{b2t} are satisfied and let $\mathcal U \subset \br \times \br \times \mathscr H^{\bm H}$ be any open bounded $G$-invariant with $\partial \mathcal U \cap \Lambda^{\bm H} = \emptyset$. If $\mathcal C \subset \mathscr S^{\bm H}$ is a branch of nontrivial solutions to \eqref{eq:H_fixed_operator_equation} bifurcating from an  critical point $(\alpha_0,\beta_0,0) \in \mathcal U \cap \Lambda^{\bm H}$, then one has the following alternative:
\begin{enumerate}[label=$(\alph*)$]
\item\label{alt_a}  either $\mathcal C \cap \partial \mathcal U \neq \emptyset$;
    \item\label{alt_b} or there exists a finite set
    \begin{align*}
        \mathcal C \cap \Lambda^{\bm K} = \{ (\lambda_0,0),(\lambda_1,0), \ldots, (\lambda_{n_0},0) \},
    \end{align*}
    satisfying the following relation
    \begin{align*}
        \sum\limits_{i=1}^{n_0}\omega_G^{\bm H}(\lambda_i) = 0.
    \end{align*}
\end{enumerate} 
\end{theorem}
In turn, we can further refine the index set \eqref{def:H_fixed_index_set}, as was done in Section \ref{sec:computation_local_bif_inv} to obtain the local equivariant bifurcation result (cf. Theorem \ref{thm:main_local_bif}),
in the following manner: 
Take a orbit type of maximal kind $(H) \in \mathfrak M_1$ and positive folding $m > 0$ with $2 \nmid m$.
Under condition \ref{b2}, we can always enumerate the the $\bm H$-fixed critical set 
\begin{align} \label{eq:H_fixed_critical_set_enumerated}
    \Lambda^{\bm H} = \{ (\alpha_i,\beta_i,0) \}_{i=1}^{\infty}, 
\end{align}
and, with respect to these indices, put
\begin{align} \label{def:H_fixed_index_set_m_folding_global}
     \Sigma_s^{\bm H} := \{ (n,j,k,i) : i \in \bn, \; (n,j,k) \in \Sigma_s(\alpha_i,\beta_i) \},
\end{align}
for any $s > 0$.
We now have all the necessary components to formulate our main global equivariant bifurcation result.
\begin{theorem}  \rm \label{thm:main_global_bif}
    Under conditions \ref{a0}--\ref{a4}, with a critical set $\Lambda^{\bm H}$ satisfying condition \ref{b2}, if there is a odd folding $s \in \bn$ and an orbit type of maximal kind $(H) \in \mathfrak M_1$ satisfying
    \begin{align} \label{eq:global_bif_nontrivial_coeff_sH}
    \sum\limits_{ (n,j,k,i) \in \Sigma_s^{\bm H}} \rho_{s,n,j,k}(\alpha_i,\beta_i) \operatorname{coeff}^{{}^{s}H}( \deg_{\mathcal V_{s,j}^{[2 \nmid n]}})  \neq 0,
    \end{align}
    then there exists a critical point $(\lambda_0,0) = (\alpha_0,\beta_0,0) \in \Lambda^{\bm H}$ and an unbounded branch of non-stationary solutions $\mathcal C \subset \mathscr S^{\bm H}$ bifurcating from $(\lambda_0,0)$ with symmetries at least $({}^{s}H)$.
\end{theorem}
\begin{proof}
Making the same arguments employed in Theorem \ref{thm:main_local_bif}, the coefficient of $({}^{s}H)$ in the sum of local bifurcation invariants associated with the enumerated $\bm H$-critical set \eqref{eq:H_fixed_critical_set_enumerated} is given by
\begin{align*}
    \operatorname{coeff}^{{}^{s}H}\left(\sum\limits_{(\lambda_i,0) \in \Lambda^{\bm H}}\omega_G^{\bm H}(\lambda_i)\right) = \sum\limits_{(\lambda_i,0) \in \Lambda^{\bm H}} \sum\limits_{ (n,j,k) \in \Sigma_s(\alpha_i,\beta_i)} \rho_{s,n,j,k}(\alpha_i,\beta_i) \operatorname{coeff}^{{}^{s}H}( \deg_{\mathcal V_{s,j}^{[2 \nmid n]}}),
\end{align*}
and the result follows from the observation 
\[
\sum\limits_{(\lambda_i,0) \in \Lambda^{\bm H}} \sum\limits_{ (n,j,k) \in \Sigma_s(\alpha_i,\beta_i)} \rho_{s,n,j,k}(\alpha_i,\beta_i) \deg_{\mathcal V_{s,j}^{[2 \nmid n]}} = \sum\limits_{ (n,j,k,i) \in \Sigma_s^{\bm H}} \rho_{s,n,j,k}(\alpha_i,\beta_i) \deg_{\mathcal V_{s,j}^{[2 \nmid n]}}.
\]
\end{proof}
\vs
In the same way Corollary \ref{cor:main_local_bif} followed from Theorem \ref{thm:main_local_bif} (cf. Remark \ref{rm:crossing_sign}), Theorem \ref{thm:main_global_bif} implies:
\begin{corollary} \label{cor:main_global_bif}
Under conditions \ref{a0}--\ref{a4}, with a critical set $\Lambda^{\bm H}$ satisfying condition \ref{b2}, if there is an odd, positive number $s \in \bn$ with
    \[
    \rho_{s,n,j,k}(\alpha_i,\beta_i) \cdot \rho_{s,n',j',k'}(\alpha_{i
'},\beta_{i'}) \geq 0
    \]
for every pair of quadruples $(n,j,k,i),(n',j',k',i') \in \Sigma_s^{\bm H}$ and an orbit type of maximal kind
$(H) \in \mathfrak M_1$ satisfying $(H) \in \mathfrak M_{1,j}^{[2 \nmid n]}$ for at least one quadruple $(n,j,k,i) \in \Sigma_s^{\bm H}$ with
$\rho_{s,n,j,k}(\alpha_i,\beta_i) \neq 0$, then there exists an unbounded branch of non-stationary solutions bifurcating from the critical point $(\lambda_i,0) = (\alpha_i,\beta_i,0) \in \Lambda^{\bm H}$ and with symmetries at least $({}^{s}H)$.
\end{corollary}
\section{A Motivating Example: Symmetric System of $N$ Vibrating Strings With Nonlinear Forces, Damping and Delay} \label{sec:n_strings}
As a preliminary motivating example, consider the vibrations of a string of length $\pi$, with both ends fastened, subjected to nonlinear forces, non-trivial damping and delay, modeled by the following boundary value problem
\begin{equation}\label{eq:1DexampleND}
\begin{cases}
\nu^2 \partial^2_t u - \partial^2_x u + \delta \partial_t u + \beta S_\tau u = u^3 +  u, \quad u(t,x) \in \br; \\
u(t,-\frac{\pi}{2})= u(t,\frac{\pi}{2}) = 0, \quad x \in [-\frac{\pi}{2},\frac{\pi}{2}], \; t \in \br^+; \\
u(t + 2\pi, x) = u(t,x),
\end{cases}
\end{equation}
where $u(t,x) \in V$, $(\alpha,\beta) \in \br \times \br$ and $(\nu,\delta,\tau) \in \bq \times \br^+ \times \br^+$ have the same interpretation they were given in Section \ref{sec:intro}. While the equations \eqref{eq:1DexampleND} faithfully capture the vibrations of a single stringed system, they are not a particularly realistic model of anything more elaborate. Indeed, our interest in strings often arises from physically integrated phenomena, such as the playing of musical instruments, which are rarely single stringed. 
\vs
A more useful model might instead require an arrangement of some number of coupled vibrating strings. For example, let $\mathbb G$ be an undirected graph, invariant under the permutation action of a finite group $\Gamma$, with $N$ vertices representing a collection of vibrating strings and with edges representing the coupling relations between vertices. Our two-parameter $(\alpha,\beta) \in \br \times \br$ model for this configuration is the system 
\begin{equation}\label{eq:NDexample}
\begin{cases}
\nu^2 \partial^2_t u - \partial^2_x u + \delta \partial_t u + \beta S_\tau u = |u|^pu - \zeta(\alpha) (L + \id)u, \quad u(t,x) \in V; \\
u(t,-\frac{\pi}{2})= u(t,\frac{\pi}{2}) = 0, \quad x \in [-\frac{\pi}{2},\frac{\pi}{2}], \; t \in \br^+; \\
u(t + 2\pi, x) = u(t,x) ,
\end{cases}
\end{equation}
where $|u|^pu := (|u_1|^pu, |u_2|^pu, \ldots, |u_N|^pu)$ for any even $p > 1$, $V:= \br^N$, $\zeta: \br \rightarrow \br$ is a continuous function related to the coupling strength between strings and $L:V \rightarrow V$ is the {\it graph Laplacian matrix} for $\mathbb G$.
\begin{remark} \rm
Since the symmetry group $\Gamma$ is left ambiguous in this section, the exact symmetric configuration of our strings is not specified. As we shall see, this does not prevent from making some definitive local and global symmetric bifurcation predictions for \ref{eq:NDexample}.
\end{remark}
Notice that the system \eqref{eq:NDexample} satisfies condition \ref{a0} by construction and that $u^3$ is a differentiable function satisfying the three conditions \ref{a1}--\ref{a3}. Notice also that, since $L$ is a $\Gamma$-equivariant, symmetric positive-semidefinite matrix, it has real, non-negative eigenvalues $z_{j,k}$ corresponding to eigenspaces $E_{j,k} \simeq \mathcal V_{j}$
such that condition \ref{a4} is satisfied for the $\Gamma$-equivariant family of matrices $A := \zeta (L + \id): \br \rightarrow L^\Gamma(V)$ with the eigenvalues
\begin{align} \label{def:zeta_jk_exampleND}
    \zeta_{j,k}(\alpha):= \zeta(\alpha) (z_{j,k} + 1), \quad j \in \{0,\ldots,r\}, \; k \in \{1,\ldots,m_j\}.
\end{align}
As a final prerequisite for using the computational formula \eqref{eq:H_fixed_local_bif_inv_comp_formula} suggested by Lemma \ref{lemm:H_fixed_local_bif_inv_comp_formula} to describe the local bifurcation invariant at each $\bm H$-fixed critical point, we must verify that the $\bm H$-fixed critical set is discrete. With this in mind, take $(\alpha_0,\beta_0,0) \in \Lambda^{\bm H}$, $(m,n,j,k) \in \Sigma^{\bm H}_0(\alpha_0,\beta_0)$ and notice that the Jacobian matrix
\begin{align} \label{eq:Jacobian_mu_mnj}
 D \mu_{m,n,j,k}(\alpha_0,\beta_0) := 
 \frac{1}{\xi_{m,n}}\begin{pmatrix*}[r]
   \partial_{\alpha}\zeta_{j,k}(\alpha_0) & \cos(m \tau) \\
    0 & -\sin(m \tau)
\end{pmatrix*}, 
 \end{align}
is nonsingular if and only if both $\partial_{\alpha}\zeta_{j,k}(\alpha_0) \neq 0$ and $m \tau \notin \pi \bz$. With this in mind, we propose the following pair of assumptions:
\begin{enumerate}[label=($D_\arabic*$)]
%\item\label{d1} the model variables $c,d \in \br \setminus \{0\}$ are such that $d > c$ and $2c > -d$; 
\item\label{d1} the function $\zeta: \br \rightarrow \br$ is differentiable and strictly monotonic;
\item\label{d2} the model variable $\tau$ is such that $\tau \notin \pi \bq$.
\end{enumerate}
Under conditions \ref{d1}--\ref{d2}, a trivial solution $(\alpha_0,\beta_0,0) \in \br \times \br \times \mathscr H$ is a $\bm H$-fixed critical point for \eqref{eq:NDexample} if and only there exists a triple $(m,n,j,k) \in \Sigma$ with $2 \nmid m$ satisfying both
\begin{align} \label{eq:critical_condition_alpha}
    \alpha_0 = \zeta^{-1}\left(\frac{\nu^2 m^2 - n^2 - \delta m \cot(m \tau)}{z_{j,k}+1} \right),
\end{align}
and
\begin{align} \label{eq:critical_condition_beta}
    \beta_0 = \frac{\delta m}{\sin(m \tau)}.
\end{align}
\begin{lemma} \rm \label{lemm:isolated_cpts_example3D}
Under conditions \ref{d1}--\ref{d2} every $\bm H$-fixed critical set $\Lambda^{\bm H}$ is discrete.
\end{lemma}
\begin{proof}
Take $(\lambda_0,0) = (\alpha_0,\beta_0,0) \in \Lambda^{\bm H}$. Since the eigenvalues \eqref{def:mu_mnj_example} form an equicontinuous family, there exists a uniform $\hat \varepsilon > 0$ such that, for any $(m,n,j,k) \in \Sigma \setminus \Sigma^{\bm H}_0(\alpha_0,\beta_0)$, one has
\begin{align}
       \mu_{n,m,j,k}(\alpha,\beta) \neq 0,
\end{align}
for all $(\lambda,0) = (\alpha,\beta,0)$ with $0 < | \lambda - \lambda_0 | \leq \hat \varepsilon$. On the other hand, by the inverse function theorem, $(\alpha_0,\beta_0)$ is an isolated zero of every $\mu_{m,n,j,k}$ with $(m,n,j,k) \in \Sigma^{\bm H}_0(\alpha_0,\beta_0)$ and the result follows from the from finiteness of $\Sigma^{\bm H}_0(\alpha_0,\beta_0)$.
\end{proof}
\vs
Now that satisfaction of the conditions \ref{a0}--\ref{a4} and
the discreteness of $\Lambda^{\bm H}$ have been demonstrated, the local bifurcation invariant at any $\bm H$-fixed critical point $(\alpha_0,\beta_0,0) \in \Lambda^{\bm H}$ becomes
\begin{align}\label{eq:local_bif_inv_comp_formula_example}
        \omega_G(\lambda_0) =  \sum\limits_{(m,n,j,k) \in \Sigma_0^{\bm H}(\alpha_0,\beta_0)} \rho_{m,n,j,k}(\alpha_0,\beta_0) \deg_{\mathcal V_{m,j}^{[2 \nmid n]}},
    \end{align}
where
\begin{align}\label{def:local_bif_inv_coeff_example}
    \rho_{m,n,j,k}(\alpha_0,\beta_0) 
    := \deg(\mu_{m,n,j,k}(\alpha,\beta), B_{\varepsilon}(\lambda)).
\end{align}
Moreover, since $(\alpha_0,\beta_0) \in \br \times \br$ is a regular value for any $\mu_{m,n,j,k}$ with $(m,n,j,k) \in \Sigma^{\bm H}_0(\alpha_0,\beta_0)$, the Brouwer degrees \eqref{def:local_bif_inv_coeff_example} can be computed using the formula
\begin{align} \label{eq:rho_mnj_formula_exampleND}
    \rho_{m,n,j,k}(\alpha_0,\beta_0) 
    &= 
    \operatorname{sign}  \det  D \mu_{m,n,j,k}(\alpha_0,\beta_0) \\
 &= \operatorname{sign}  \det
 \begin{pmatrix*}[r]
   \partial_{\alpha}\zeta_{j,k}(\alpha_0) & \cos(m \tau) \\
    0 & -\sin(m \tau)
\end{pmatrix*} = \operatorname{sign}-\partial_{\alpha}\zeta_{j,k}(\alpha_0)\sin(m \tau). \nonumber
 \end{align}
We are now in a position to present the main local and global bifurcation results for this example. First, let's apply Theorem \ref{thm:main_local_bif} to detect the branches of non-stationary solutions to system \eqref{eq:NDexample} emerging from its $\bm H$-fixed critical points.
\begin{proposition} \label{prop:NDexample_local_bifurcation}
Under conditions \ref{d1}--\ref{d2}, every index quadruple $(m,n,j,k) \in 2 \bn - 1 \times \bn \times \{0,\ldots,r\} \times \{1,\ldots, m_j\}$ is associated with a parameter pair 
\[
\alpha_{m,n,j,k} := \zeta^{-1}\left(\frac{\nu^2 m^2 - n^2 - \delta m \cot(m \tau)}{z_{j,k}+1} \right), \quad \beta_{m,n,j,k} : = \frac{\delta m}{\sin(m \tau)},
\]
which serves as a branching point for a branch of non-stationary solutions to \eqref{eq:NDexample} with symmetries at least $({}^{m}H)$ for each orbit type of maximal kind $(H) \in \mathfrak M_{1,j}^{[2 \nmid n]}$. 
\end{proposition}
\begin{proof}
Since the values $\partial_\alpha \zeta_{j,k}(\alpha)$ are of a fixed sign for all $j \in \{0,1,\ldots,r\}$, $k \in \{1,2,\ldots,m_j\}$ and $\alpha \in \br$, the result follows from formula \eqref{eq:rho_mnj_formula_exampleND} together with Corollary \ref{cor:main_local_bif}.
\end{proof}
\vs
Finally, let's apply Theorem \ref{thm:main_global_bif} in order to establish the global properties of the branches of non-stationary solutions predicted by Proposition \eqref{prop:NDexample_local_bifurcation}.
\begin{proposition} \label{prop:NDexample_global_bifurcation}
Under conditions \ref{d1}--\ref{d2}, every branch of non-trivial solutions to \eqref{eq:NDexample} emerging from a critical point $(\alpha_0,\beta_0,0) \in \Lambda$ satisfying 
\[
\alpha_{0} := \zeta_{j,k}^{-1}(\nu^2 m^2 - n^2 - \delta m \cot(m \tau)) \text{ and } \beta_{0} : = \frac{\delta m}{\sin(m \tau)},
\]
for some $(m,n,j,k) \in 2 \bn - 1 \times \bn \times \{0,\ldots,r\} \times \{1,\ldots, m_j\}$ consists only of non-stationary solutions and is also unbounded.
\end{proposition}
\begin{proof}
Again, the result follows from the fact that all the values $\partial_\alpha \zeta_{j,k}(\alpha)$ are of a fixed sign, together with formula \eqref{eq:rho_mnj_formula_exampleND} and Corollary \ref{cor:main_global_bif}.
\end{proof}
\begin{remark} \rm
  Theorem \ref{thm:intro_arbitrary_Gamma} follows as a direct corollary of the Propositions \ref{prop:NDexample_local_bifurcation} and \ref{prop:NDexample_global_bifurcation}.   
\end{remark}
%%%%%%%%%%%%%%%%%%%%%%%%%%
%%%%%%%%%%%%%%%%%%%%%%%%%%
%%%%%%%%%%%%%%%%%%%%%%%%%%
\section{The Special Case of Dihedral Symmetries} \label{sec:example_2}
\epigraph{Of a musical string, of given length and weight, stretched by a given weight, to
find its vibrations.}{Bernoulli (cf. \cite{Struik})}
%%%%%%%%%%%%%%%%%%%%%%%%%%%%%%%%
%%%%%%%%%%%%%%%%%%%%%%%%%%%%%%%%
%\begin{figure}[H]
%\begin{center}
%\vskip3cm
%\rput(0,0){\psline(2.6,-1.5)(2.6,1.5)}
%\rput{60}(0,0){\psline(2.6,-1.5)(2.6,1.5)}
%\rput{120}(0,0){\psline(2.6,-1.5)(2.6,1.5)}
%\rput{180}(0,0){\psline(2.6,-1.5)(2.6,1.5)}
%\rput{-60}(0,0){\psline[linestyle=dashed](2.6,-1.5)(2.6,1.5)}
%\rput{-120}(0,0){\psline(2.6,-1.5)(2.6,1.5)}
%\rput{60}(0,0){\rput(0,3){\psdots[linewidth=14pt](0,0)\rput{-60}(0,0){\scriptsize \white $u_2$}}}
%\rput{0}(0,0){\rput(0,3){\psdots[linewidth=14pt](0,0)\rput(0,0){\scriptsize \white $u_1$}}}
%\rput{120}(0,0){\rput(0,3){\psdots[linewidth=14pt](0,0)\rput{-120}(0,0){\scriptsize \white $u_3$}}}
%\rput{-60}(0,0){\rput(0,3){\psdots[linewidth=14pt](0,0)\rput{60}(0,0){\scriptsize \white $u_{N}$}}}
%\rput{-120}(0,0){\rput(0,3){\psdots[linewidth=14pt](0,0)\rput{120}(0,0){\scriptsize \white $u_{N-1}$}}}
%\rput{180}(0,0){\rput(0,3){\psdots[linewidth=14pt](0,0)\rput{180}(0,0){\scriptsize \white $u_4$}}}
%\end{center}
%\vskip3cm
%\caption{Cycle of $N$ Vibrating Strings with $\Gamma=D_{N}$-symmetries}\label{fig:chain}
\begin{figure}
    \centering
\includegraphics[width=0.5\linewidth]{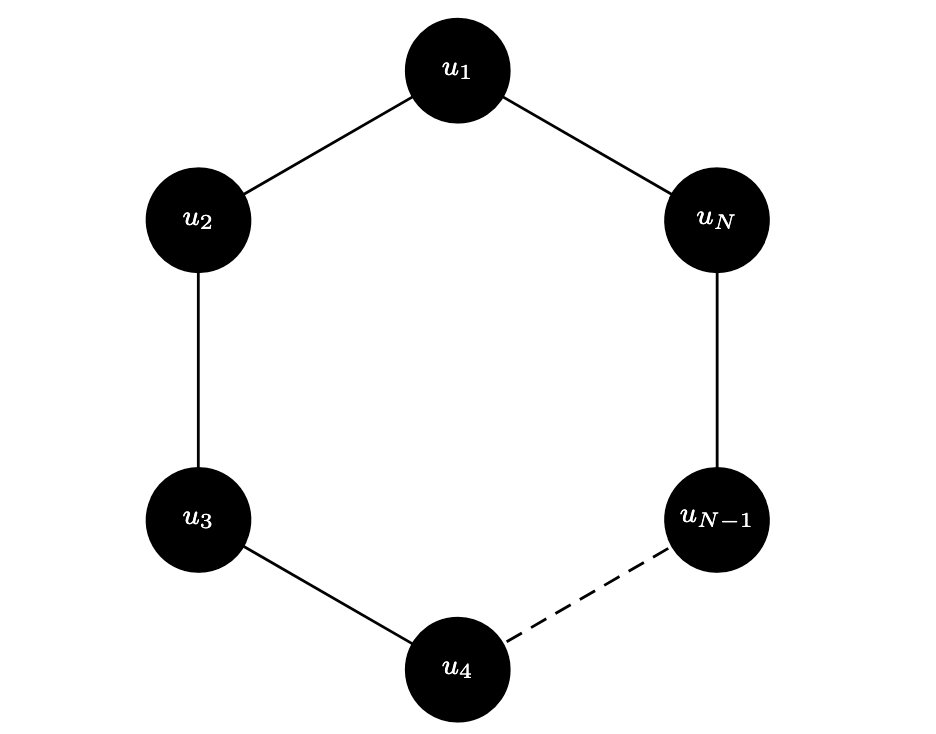}
    \caption{Cycle of $N$ Vibrating Strings with $\Gamma=D_{N}$ symmetries}
    \label{fig:dihedral_graph}
\end{figure}
%%%%%%%%%%%%%%%%%%%%%%%%%%%%%%%%
%%%%%%%%%%%%%%%%%%%%%%%%%%%%%%%%
In this section, we are able to offer a more definitive diagnosis on the symmetries expressed by the branches of non-stationary solutions to \eqref{eq:NDexample} emerging from its critical points after prescribing an exact coupling relation between strings in our system. Specifically, let $\Gamma$ be the dihedral group of order $2N$ such that our full symmetry group becomes
\[
G := S^1 \times \bz_2 \times \bz_2 \times D_N,
\]
and consider the natural permutation representation $\rho_V: \Gamma \rightarrow GL(V)$ on $V$ defined for any $\sigma \in \Gamma$ and $u \in V$ by
\begin{align} 
\rho_V(\sigma)(u_1,u_2,\ldots,u_N) := (u_{\sigma(1)},u_{\sigma(2)},\ldots,u_{\sigma(N)}). 
\end{align}
In particular, the rotation and reflection generators $\gamma, \kappa \in \Gamma$ are the permutations of the indices $i \in \{1,\ldots,N\}$ given by 
\[
\gamma(i):= i+1 \Mod{N} \text{ and } \kappa(i) := N-i \Mod{N},
\]
and with the characters 
\begin{align*}
    \chi_V(\gamma) = 0 \; \text{  and  } \; \chi_V(\kappa) = 
    \begin{cases}
        1 & \text{ if } N \text{ is  odd}; \\
        0 & \text{ if } N  \text{ is  even}.
    \end{cases}
\end{align*}
The number and character of the irreducible $\Gamma$ representations also depend on the dihedral order $N$: there are always the {\it trivial representation} $\mathcal V_0 \simeq \br$, on which $\Gamma$ acts trivially, $\lfloor \frac{N+1}{2} \rfloor - 1$ {\it geometric representations} $\mathcal V_j$, each with an action induced by the corresponding matrix representation $\rho_j: \Gamma \rightarrow GL(\bc)$
    \[
\rho_j(\kappa) := \begin{pmatrix}
    1&0 \\
    0&-1
\end{pmatrix}, \quad \rho_j(\gamma):=\begin{pmatrix}
    \cos(\frac{2j\pi}{N}) & -\sin(\frac{2j\pi}{N}) \\
    \sin(\frac{2j\pi}{N}) & \cos(\frac{2j\pi}{N})
\end{pmatrix}, \quad 1 \leq j < \lfloor \frac{N+1}{2} \rfloor,
\]
the {\it sign representation} $\mathcal V_{*} \simeq \br$, with the action
\[
\sigma \cdot v := \sign(\sigma) v, \quad v \in \mathcal V_{*},
\]
and, in the case that $N$ is even, two additional irreducible one-dimensional representations $\mathcal V_{\lfloor \frac{N+1}{2} \rfloor}, \mathcal V_{**} \simeq \br$ equipped, respectively, with the actions
\[
\sigma \cdot v := - \sign(\sigma) v, \quad v \in \mathcal V_{\lfloor \frac{N+1}{2} \rfloor},
\]
and
\[
\rho_{**}(\kappa) = \rho_{**}(\gamma) = -1.
\]
Comparing characters for the irreducible representations of $\Gamma$ with the character of $V$
\begin{table}[h]
\centering
\begin{tabular}{|c|ccccc|}
\hline
conjugacy classes &$e_\Gamma$ & \; & $\kappa$ & \; &$\gamma$\\\hline 
$\chi_0$ &$1$ & \; & $1$ & \; & $1$ \\
$\chi_1$ &$2$ & \; & $0$ & \; & $\cos(\frac{2 \pi}{N})+\cos(\frac{2 \pi}{N})$\\
$\vdots$ & $\vdots$ & \;& $\vdots$ & \; & $\vdots$ \\
$\chi_j$ &$2$ & \; & $0$ & \; & $\cos(\frac{2 j \pi}{N})+\cos(\frac{2 j\pi}{N})$\\
$\vdots$ & $\vdots$ & \;& $\vdots$ & \; & $\vdots$ \\
$\chi_{\lfloor \frac{N+1}{2} \rfloor}$ &$1$ & \; & $-1$ & \; & $1$\\
$\chi_{*}$ &$1$ & \; & $1$ & \; & $-1$\\
$\chi_{**}$ &$1$ & \; & $-1$ & \; & $-1$\\
\hline 
$\chi_V$ &$N$ & \; & $\frac{1}{2}(1-(-1)^N)$ & \; & $0$ \\
\hline 
\end{tabular}
\vs
\caption{Character Table for $D_N$.}
\end{table} \\
one obtains the relation
\begin{align}
    \chi_V = 
    \begin{cases}
        \chi_0 + \chi_1 + \cdots + \chi_{\frac{N+1}{2} - 1} &  \text{if } 2 \nmid N; \\
        \chi_0 + \chi_1 + \cdots + \chi_{\frac{N}{2}-1} + \chi_{\frac{N}{2}} & \text{if } 2 \mid N,
    \end{cases}
\end{align}
implying that $V$ has the $\Gamma$-isotypic decomposition 
\begin{align}
    V = \begin{cases}
        V_0 \oplus V_1 \oplus \cdots \oplus V_{\frac{N+1}{2} - 1} & \text{if } 2 \nmid N; \\
       V_0 \oplus V_1 \oplus \cdots \oplus V_{\frac{N}{2}-1} \oplus V_{\frac{N}{2}} & \text{if } 2 \mid N.
    \end{cases}
\end{align}
For the sake of generality, we adopt the following notation for the set of $\Gamma$-isotypic indices relevant to the $\Gamma$-isotypic decomposition of $V$
\[
\mathfrak J(N) := \begin{cases}
    \{0,1,\ldots, \frac{N+1}{2} - 1\} &  \text{if } 2 \nmid N; \\
    \{0,1,\ldots, \frac{N}{2} - 1, \frac{N}{2}\} &  \text{if } 2 \mid N.
\end{cases}
\]
The graph Laplacian associated with the undirected graph $\mathbb G$ with $N$ vertices invariant under the permutation action of $D_N$ has the form
\[
L=\left(
\begin{array}
[c]{ccccccc}%
-2 & 1 & 0 & \dots & 0 & 0 & 1\\
1 & -2 & 1 & \dots & 0 & 0 & 0\\
0 & 1 & -2 & \dots & 0 & 0 & 0\\
\vdots & \vdots & \vdots & \ddots & \vdots & \vdots & \vdots\\
0 & 0 & 0 & \dots & -2 & 1 & 0\\
0 & 0 & 0 & \dots & 1 & -2 & 1\\
1 & 0 & 0 & \dots & 0 & 1 & -2
\end{array}
\right).
\]
One can verify that such a matrix has the eigenvectors
\[
v_j := (1,\gamma^j,\gamma^{2j},\ldots,\gamma^{(N-1)j})^T, \; j \in \{0,1,\ldots,\lfloor \frac{N+1}{2} \rfloor - 1\},
\]
where $\gamma := e^{\frac{2\pi i}{N}}$, corresponding to the eigenvalues
\[
-z_j^2 := - 2 + \gamma^j + \gamma^{-j} = -4 \sin^2 \left( \frac{\pi j}{N} \right).
\]
And, in the case that $N$ is even, one will find that there is also the eigenpair
\[
v_{\frac{N}{2}} := (1,-1,1\ldots, -1)^T, \quad z_{\frac{N}{2}} = -4.
\]
Since each isotypic component of $V$ has simple isotypic multiplicity, i.e. since
\[
m_j = 1, \quad j \in \mathfrak J(N),
\]
the eigenvalues for the family of $\Gamma$-equivariant matrices $\zeta (L + \id): \br \rightarrow L^\Gamma(V)$ are specified with a single index (rather than an index pair, cf. \eqref{eq:Gamma_Z2Z2_block_matrix_decomp})
\begin{align} \label{def:zeta_j_example}
    \zeta_j(\alpha):= \zeta(\alpha) (z_j + 1), \quad j \in \mathfrak J(N),
\end{align}
and the eigenvalues for the family of $G$-equivariant linear operators $\mathscr A: \br \times \br \rightarrow L^G(\mathscr H)$ are specified with an index triple (rather than an index quadruple, cf. \eqref{def:eigenvalues_Anmjk})
\begin{align} \label{def:mu_mnj_example}
\mu_{m,n,j}(\alpha,\beta) := \frac{-\nu^2 m^2 + n^2 + i \delta m + \zeta_j(\alpha) + \beta e^{-im \tau}}{\xi_{m,n}}, \quad (m,n,j) \in \bn \cup \{0\} \times \bn \times \mathfrak J(N).
\end{align}
For compatibility with the setting of simple isotypic multiplicities, we can redefine the index set \eqref{def:index_set_full} by
\begin{align} \label{def:index_set_full_example}
    \Sigma := \{ (m,n,j) : m \in \bn \cup \{0\}, \; n \in \bn, \; j \in \mathfrak J(N) \},
\end{align}
and also the index sets \eqref{def:index_set_null} and \eqref{def:H_fixed_index_set}, respectively, by
\begin{align} \label{def:index_set_null_example}
    \Sigma_0 (\alpha,\beta) := \{ (m,n,j) \in \Sigma :  m > 0, \; \mu_{m,n,j}(\alpha,\beta) = 0\},
\end{align}
and
\begin{align} \label{def:H_fixed_index_set_example3D}
    \Sigma_0^{\bm H} (\alpha,\beta) := \{ (m,n,j) \in \Sigma_0 (\alpha,\beta) : 2 \nmid m \}.
\end{align}
A natural first candidate for function $\zeta: \mathbb{R} \to \mathbb{R}$, which describes the dependence of coupling intensity on the bifurcation parameter $\alpha \in \mathbb{R}$ might be $\zeta(\alpha) = \alpha$, implying an unbounded growth of coupling strength. However, such behavior is often not physically realistic as many systems exhibit saturation limits where coupling strength stabilizes beyond a certain range. A more reasonable candidate is the sigmoid function 
\begin{align}\label{eq:sigmoid_function}
\zeta(\alpha) := \frac{1}{1 + e^{-\alpha}}.
\end{align}
\begin{figure}[htbp]
    \centering \includegraphics[width=.90\textwidth]{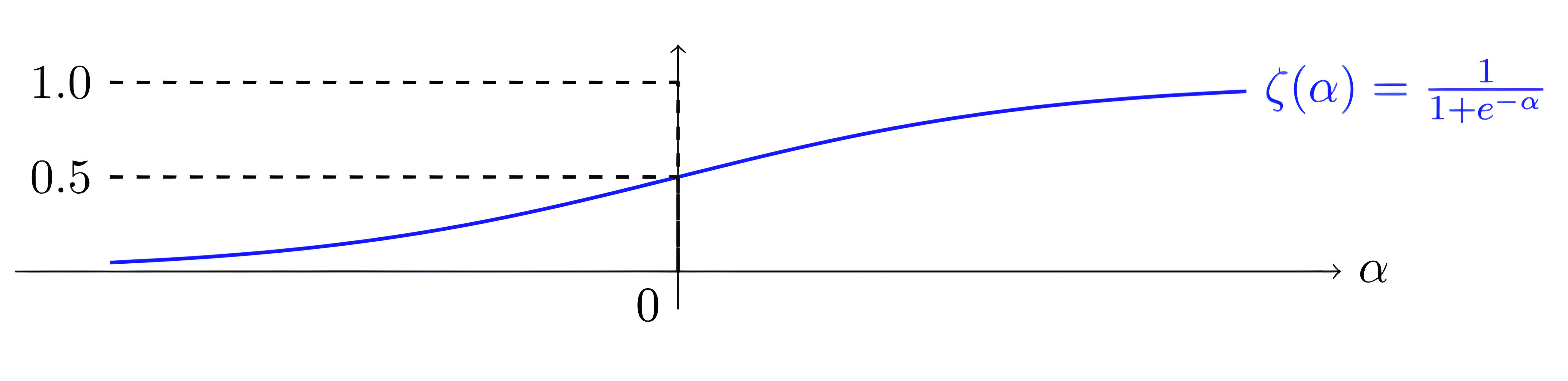}
    \caption{Graph of the sigmoid function $\zeta(\alpha)$.}
    \label{Fig-2}
\end{figure}
\begin{remark} \rm
The sigmoid function introduces a nonlinear dependence of the coupling in \eqref{eq:system}
on the bifurcation parameter $\alpha \in \mathbb{R}$ and its asymptotic behavior suggests the imposition of {\it coupling strength saturation limits,} which might be understood as the physical constraints of our system.
\end{remark}
In order to leverage our specification of the coupling group $\Gamma$ for predicting the symmetric properties of the branches of solutions emerging from the critical points to our system, we must explicitly describe the orbit types belonging to the sets $\mathfrak M_{m,j}^{[2 \nmid n]}$, i.e. the maximal elements of the orbit type lattices $\Phi_1^t(G; \mathscr H_{m,j}^{[2 \nmid n]} \setminus \{0\})$ for every $m > 0$, $j \in \{0,1,\ldots,r\}$, $n \in \bn$ and for both even and odd dihedral orders $N$. Since the invariant theory for $S^1 \times D_N$ is already described by Golubitsky in \cite{Gol}, we only need to consider the impact of an extra two copies of $\bz_2$ on the orbit types of maximal kind and modify his notation to fit our framework. 
\vs
There is a class of subgroup, which we denote by $H_{m,j}^{[2 \nmid n]} \leq S^1 \times \bz_2 \times \bz_2 \times D_N$, with a conjugacy class that is maximal in all orbit type lattices $\Phi_1^t(G; \mathscr H_{m,j}^{[2 \nmid n]} \setminus \{0\})$ with $m > 0$, $j \in \{0,1,\ldots,r\}$, $n \in \bn$ and for any dihedral parity (except in the case of $j = \frac{N}{2}$ where $N$ must be even), generated by the elements
\begin{align*}
    \begin{cases}
%\{ (e^{\frac{\pi}{m}},-1,1,1), (1, 1,-1,1), (1, 1,1,\gamma), (1, 1,1,\kappa) \} & \text{ if } j = 0 \text{ and } n \text{ is  even}; \\
\{(e^{\frac{\pi}{m}},-1,1,1), (1,(-1)^n,-1,1),  (1, 1,1,\gamma), (1, 1,1,\kappa) \} & \text{ if } j = 0; \\
%\{(e^{\frac{\pi}{m}},-1,1,1), (1,1,-1,1), (e^{\frac{\pi}{m}}, 1,1,\gamma)\} & \text{ if } j = \lfloor \frac{N+1}{2} \rfloor +1 \text{ and } n \text{ is  even}; \\
\{(e^{\frac{\pi}{m}},-1,1,1), (1,(-1)^n,-1,1),  (e^{\frac{\pi}{m}}, 1,1,\gamma)\} & \text{ if } j = \frac{N}{2}; \\
%\{(e^{\frac{\pi}{m}},-1,1,1), (1,1,-1,1),  (e^{\frac{-2\pi}{N}j}, 1,1,\gamma)\}, & \text{ if } 1 \leq j \leq \lfloor \frac{N+1}{2} \rfloor \text{ and } n \text{ is  even}; \\
\{(e^{\frac{\pi}{m}},-1,1,1), (1,(-1)^n,-1,1), (e^{\frac{-2\pi}{N} j}, 1,1,\gamma)\}, & \text{ if } 0 < j < \lfloor \frac{N+1}{2} \rfloor.
    \end{cases}
\end{align*}
Let's examine the possible generators of these subgroups and their implications for the functions $u \in \mathscr H$ with $(G_u) \geq (H_{m,j}^{[2 \nmid n]})$. The generator $(e^{\frac{\pi}{m}},-1,1,1)$ implies $\frac{\pi}{m}$-anti-periodicity and the generators $(1,1,-1,1), (1,-1,-1,1)$ imply, respectively,  evenness and oddness, such that any function $u \in \mathscr H$ with $(G_u) \geq (H_{m,j}^{[2 \nmid n]})$ satisfies the symmetry relations
\[
u(t + \frac{\pi}{m},x) = -u(t,x) \text{ and } u(-t,x) = (-1)^n u(t,x), \quad \forall_{(t,x) \in \Om}.
\]
On the other hand, presence of the generators $(1, 1,1,\gamma), (1, 1,1,\kappa)$ implies the subgroup inclusion $\{1\} \times \{1\} \times \{1\} \times D_N \leq H^{[2 \nmid n]}_{m,0}$ which, in turn, indicates total permutation invariance such that, for any $u \in \mathscr H$ with $(G_u) \geq (H^{[2 \nmid n]}_{m,0})$, one has 
\[
 \sigma u(t,x) = u(t,x), \quad \forall_{\sigma \in \Gamma}, \; \forall_{(t,x) \in \Om}.
\]
Finally, the generators $(e^{\frac{\pi}{m}}, 1,1,\gamma)$, $(e^{\frac{-2\pi}{\tilde N} j}, 1,1,\gamma)$ imply the traveling wave symmetries (cf. Figure \ref{fig:eigenfunctions})
\[
u_{i}(t+  \frac{\pi}{m},x) = u_{i+1\Mod{N}}(t,x), \quad \forall_{(t,x) \in \Om}, \; i \in \{1,\ldots, N\}, 
\]
for all $u \in \mathscr H$ with $(G_u) \geq (H^{[2 \nmid n]}_{m,\frac{N}{2}})$ and
\[
u_{i}(t - \frac{-2\pi}{ N} j,x) = u_{i+1\Mod{N}}(t,x), \quad \forall_{(t,x) \in \Om}, \; i \in \{1,\ldots, N\}, 
\]
for all $u \in \mathscr H$ with $(G_u) \geq (H^{[2 \nmid n]}_{m,j})$ where $0 < j < \lfloor \frac{N+1}{2} \rfloor$. 
\vs
In order to provide descriptions for the remaining orbit types of maximal kind, some additional notations must be introduced. Given any dihedral order $N$ and $\Gamma$-isotypic index $0 < j < \lfloor \frac{N+1}{2} \rfloor$, we put
\[
\tilde N := \frac{N}{\gcd(N,j)}, \quad \tilde j:= \frac{j}{\gcd(N,j)}, 
\]
and we denote by $h$ the modular inverse of $\tilde j$ with respect to $\tilde N$, i.e. 
\[
h \tilde j = 1 \Mod{\tilde{N}}.
\]
Now, for the cases with $0 < j  < \lfloor \frac{N+1}{2} \rfloor$, the lattices $\Phi_1^t(G; \mathscr H_{m,j}^{[2 \nmid n]}\setminus \{0\}) $ admit two additional 
maximal elements, arising as the conjugacy classes associated with two types of isotropy subgroup, which we denote by $S_{m,j}^{[2 \nmid n]}$ and $T_{m,j}^{[2 \nmid n]}$. The first of these subgroups $S_{m,j}^{[2 \nmid n]} \leq S^1 \times \bz_2 \times \bz_2 \times D_N$ is generated by the elements
\begin{align*}
    \begin{cases}
\{(e^{\frac{\pi}{m}},-1,1,1), (1,(-1)^n,-1,1),  (1,1,1,\kappa), (1, 1,1,\gamma^{\tilde N}) \} & \text{ if } \tilde N \text{ is odd}; \\
\{(e^{\frac{\pi}{m}},-1,1,1), (1,(-1)^n,-1,1), (1,1,1,\kappa), (e^{\frac{\pi}{m}}, 1,1,\gamma^{\frac{\tilde N}{2} h})\} & \text{ if } \tilde N \text{ is even},
    \end{cases}
\end{align*}
while the second $T_{m,j}^{[2 \nmid n]} \leq S^1 \times \bz_2 \times \bz_2 \times D_N$ has the generators
\begin{align*}
    \begin{cases}
\{(e^{\frac{\pi}{m}},-1,1,1), (1,(-1)^n,-1,1), (e^{\frac{\pi}{m}},1,1, \kappa), (1, 1,1,\gamma^{\tilde N})\} & \text{ if } \tilde N \text{ is odd}; \\
\{(e^{\frac{\pi}{m}},-1,1,1), (1,(-1)^n,-1,1), (1,1,1,\kappa \gamma^h), (e^{\frac{\pi}{m}}, 1,1,\gamma^{\frac{\tilde N}{2} h})\}  & \text{ if } \tilde N \text{ is even}.
    \end{cases}
\end{align*}
Let's examine the the generators of $S_{m,j}^{[2 \nmid n]}$ and $T_{m,j}^{[2 \nmid n]}$ which do not appear in the subgroup $H_{m,j}^{[2 \nmid n]}$. First, the generators $(1,1,1,\kappa)$, $(1,1,1,\kappa \gamma^h)$ and $(e^{\frac{\pi}{m}},1,1, \kappa)$ imply the permutation invariance 
\[
\kappa u(t,x) = u(t,x), \quad \forall_{(t,x) \in \Om},
\]
for all $u \in \mathscr H$ with $(G_u) \geq (S_{m,j}^{[2 \nmid n]})$ the invariance
\[
\kappa \gamma^h u(t,x) = u(t,x), \quad \forall_{(t,x) \in \Om},
\]
for all $u \in \mathscr H$ with $(G_u) \geq (T_{m,j}^{[2 \nmid n]})$ where $\tilde N$ is even and
\[
u_i(t+\frac{\pi}{m},x) = u_{N-i \Mod{N}}(t,x), \quad \forall_{(t,x) \in \Om}, \; i \in \{1,\ldots, N\},
\]
for all $u \in \mathscr H$ with $(G_u) \geq (T_{m,j}^{[2 \nmid n]})$ where $\tilde N$ is odd. On the other hand, the generators $(e^{\frac{\pi}{m}},1,1,\gamma^{\frac{\tilde N}{2}h})$ and $(e^{\frac{\pi}{m}},1,1,\gamma^{\tilde N h})$ imply a third class of travelling wave (cf. Figure \ref{fig:eigenfunctions})
\[
u_{i}(t+  \frac{\pi}{m},x) = u_{i + \frac{\tilde N}{2}h \Mod{N}}(t,x), \quad \forall_{(t,x) \in \Om}, \; i \in \{1,\ldots, N\}, 
\]
for all $u \in \mathscr H$ with $(G_u) \geq (S_{m,j}^{[2 \nmid n]})$ or $(G_u) \geq (T_{m,j}^{[2 \nmid n]})$ where $\tilde N$ is even and 
\[
u_{i}(t+  \frac{\pi}{m},x) = u_{i + \tilde N h \Mod{N}}(t,x), \quad \forall_{(t,x) \in \Om}, \; i \in \{1,\ldots, N\}, 
\]
for all $u \in \mathscr H$ with $(G_u) \geq (S_{m,j}^{[2 \nmid n]})$ or $(G_u) \geq (T_{m,j}^{[2 \nmid n]})$ where $\tilde N$ is odd.
\vs
With generator descriptions for the maximal elements of the isotropy lattices 
$\Phi_1^t(G; \mathscr H_{m,j}^{[2 \nmid n]} \setminus \{0\})$ with $(m,n,j) \in \Sigma$ and $m > 0$ according to the $\Gamma$-isotypic index $j  \in \{0,1,\ldots,r\}$ and the parity of the dihedral order $N$, we can strengthen the claims made in Theorem \ref{thm:intro_arbitrary_Gamma}, as follows:
\begin{theorem} \label{thm:intro_dihedral} \rm
The trivial solution to \eqref{eq:system} with the assignments $\Gamma:= D_N$, $A(\alpha):= \zeta(\alpha)(L+\id)$, $\zeta: \br \rightarrow \br$ differentiable, strictly monotonic
and under the conditions $\nu \in \bq$, $\delta > 0$, $\tau \neq \pi \bq$, undergoes a global bifurcation at the critical parameter values (if they exist)
\[
\alpha_{m,n,j} := \zeta^{-1}\left(\frac{\nu^2 m^2 - n^2 - \delta m \cot(m \tau)}{z_{j}+1} \right), \quad \beta_{m,n,j} : = \frac{\delta m}{\sin(m \tau)},
\]
for every index triple $(m,n,j) \in 2 \bn - 1 \times \bn \times \{0,\ldots,r\}$. In particular, every trivial solution of the form
$(\alpha_{m,n,j}, \beta_{m,n,j},0)$ is a branching point for an unbounded branch of non-stationary solutions with symmetries at least $(H_{m,j}^{[2 \nmid n]})$ and, in the case that, $0 < j < \lfloor \frac{N+1}{2} \rfloor$, also a branching point for two additional unbounded branches of non-stationary solutions with symmetries at least $(T_{m,j}^{[2 \nmid n]})$ and $(S_{m,j}^{[2 \nmid n]})$.
\end{theorem}

\begin{figure}
    \centering
\includegraphics[width=.95\linewidth]{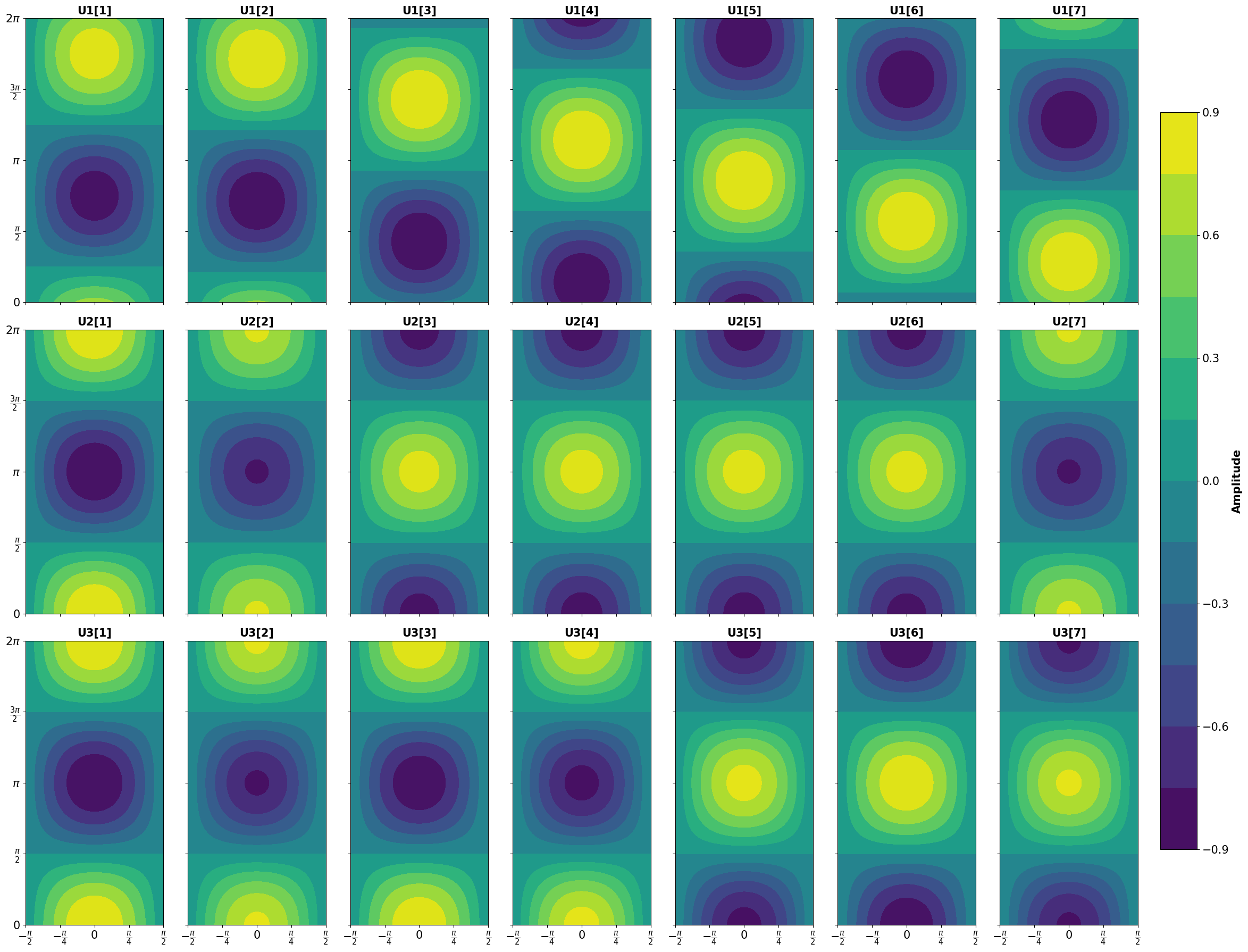}
    \caption{For the Dihedral symmetry group $\Gamma = D_7$, the functions $U_1(t,x) = \cos(x)(\cos(t)\operatorname{Re}(v_1) - \sin(t)\operatorname{Im}(v_1))$, $U_2(t,x) = \cos(x)\cos(t)\operatorname{Re}(v_1)$ and $U_3(t,x) = \cos(x)\cos(t)\operatorname{Im}(v_1)$ exhibit symmetries at least $(H_{1,1}^0)$, $(T_{1,1}^0)$ and $(S_{1,1}^0)$,
 respectively}
\label{fig:eigenfunctions}
\end{figure}
%%%%%%%%%%%%%%%%%%%%%%%%%%
%%%%%%%%%%%%%%%%%%%%%%%%%%
%%%%%%%%%%%%%%%%%%%%%%%%%%
\newpage

\appendix
\section{The $G$-Equivariant Degree}\label{sec:appendix_eqdeg}
\noi{\bf  Equivariant notation.}
Let $G$ be a compact Lie group. For any subgroup  $H \leq G$, we denote by $(H)$ its conjugacy class,
by $N(H)$ its normalizer by $W(H):=N(H)/H$ its Weyl group in $G$. The set of all subgroup conjugacy classes in $G$ is denoted by $\Phi(G):=\{(H): H\le G\}$ and has a natural partial order defined as follows
\[
(H)\leq (K) \iff \exists_{ g\in G}\;\;gHg^{-1}\leq K.
\]
As is possible with any partially ordered set, we extend the natural order over $\Phi(G)$ to a total order, which we indicate by $<$ to differentiate the two relations. Moreover, we put $\Phi_0 (G):= \{ (H) \in \Phi(G) \; : \; \text{$W(H)$  is finite}\}$ and, for any $(H),(K) \in \Phi_0(G)$, we denote by $n(H,K)$ the number of subgroups $\tilde K \leq G$ with $\tilde K \in (K)$ and $H \leq \tilde K$.
\vs
Given a $G$-space $X$ with an element $x \in X$, we denote by
$G_{x} :=\{g\in G:gx=x\}$ the {\it isotropy group} of $x$
and we call $(G_{x}) \in \Phi(G)$  the {\it orbit type} of $x \in X$. Put $\Phi(G,X) := \{(H) \in \Phi_0(G)  : 
(H) = (G_x) \; \text{for some $x \in X$}\}$ and  $\Phi_0(G,X):= \Phi(G,X) \cap \Phi_0(G)$. For a subgroup $H\leq G$, the subspace $
X^{H} :=\{x\in X:G_{x}\geq H\}$ is called the {\it $H$-fixed-point subspace} of $X$. If $Y$ is another $G$-space, then a continuous map $f : X \to Y$ is said to be {\it $G$-equivariant} if $f(gx) = gf(x)$ for each $x \in X$ and $g \in G$.
\vs
\noi{\bf The Burnside Ring and Axioms of Equivariant Brouwer Degree.}
The free $\mathbb{Z}$-module $A(G) := \mathbb{Z}[\Phi_0(G)]$ has a natural ring structure when equipped with the multiplicative operation defined as follows
\begin{align} \label{def:burnside_product}
    (H) \cdot (K) := \sum\limits_{(L) \in \Phi_0(G)} n_L(L), 
\end{align}
for any pair of generators $(H),(K) \in \Phi_0(G)$, and where the coefficients $n_L \in \mathbb{Z}$ are given by the recurrence formula
\begin{align} \label{def:recurrence_formula_coefficients_burnside_product}
    n_L := \frac{n(L,H) |W(H)| n(L,K) |W(K)| - \sum_{(\tilde L) > (L)} n_{\tilde L} n(L,\tilde L) |W(\tilde L)|}{|W(L)|}.
\end{align}
Any {\it Burnside Ring} element $a \in A(G)$ can be expressed as a formal sum over some finite number of generator elements  
\[
a = n_1(H_1) + n_2(H_2) + \cdots + n_N(H_N),
\]
and we use the notation
\[
\operatorname{coeff}^H(a) = n_H,
\]
to specify the integer coefficient standing next to the generator element $(H) \in \Phi_0(G)$.
\vs
Let $V$ be an orthogonal $G$-representation with an open bounded $G$-invariant set $\Om \subset V$. A $G$-equivariant map $f:V \rightarrow V$ is said to be $\Om$-admissible if $f(x) \neq 0$ for all $x \in \partial \Om$, in which case the pair $(f,\Om)$ is called an admissible $G$-pair in $V$. We denote by $\mathcal M^G(V)$ the set of all admissible $G$-pairs in $V$ and by $\mathcal{M}^{G}$ the set of all admissible $G$-pairs defined by taking a union over all orthogonal $G$-representations, i.e.
\[
\mathcal M^G := \bigcup\limits_V \mathcal M^G(V).
\]
The $G$-equivariant Brouwer degree provides an algebraic count of solutions, according to their symmetric properties, to equations of the form
\[
f(x) = 0, \; x \in \Omega,
\]
where $(f, \Omega) \in \mathcal M^G$. In fact, it is standard (cf. \cite{AED}, \cite{book-new}) to define the {\it $G$-equivariant Brouwer degree} as the unique map associating to every admissible $G$-pair $(f,\Om)\in \mathcal M^G$ an element from the Brouwer Ring $A(G)$, satisfying the four {\it degree axioms} of existence, additivity, homotopy and normalization:
\vs
\begin{theorem} \rm
	\label{thm:GpropDeg} There exists a unique map $\gdeg:\mathcal{M}
	^{G}\to A(G)$, that assigns to every admissible $G$-pair $(f,\Omega)$ the Burnside Ring element
	\begin{equation}
		\label{eq:G-deg0}\gdeg(f,\Omega)=\sum_{(H) \in \Phi_0(G)}%
		{n_{H}(H)},
	\end{equation}
	satisfying the following properties:
	\begin{itemize}
		\item[] \textbf{(Existence)} If  $n_{H} \neq0$ for some $(H) \in \Phi_0(G)$ in \eqref{eq:G-deg0}, then there
		exists $x\in\Omega$ such that $f(x)=0$ and $(G_{x})\geq(H)$.

		\item[] \textbf{(Additivity)} 
  For any two  disjoint open $G$-invariant subsets
  $\Omega_{1}$ and $\Omega_{2}$ with
		$f^{-1}(0)\cap\Omega\subset\Omega_{1}\cup\Omega_{2}$, one has
		\begin{align*}
			\gdeg(f,\Omega)=\gdeg(f,\Omega_{1})+\gdeg
			(f,\Omega_{2}).
		\end{align*}

		\item[] \textbf{(Homotopy)} For any 
  $\Omega$-admissible $G$-homotopy, $h:[0,1]\times V\to V$, one has
		\begin{align*}
			\gdeg(h_{t},\Omega)=\mathrm{constant}.
		\end{align*}

		\item[] \textbf{(Normalization)}
  For any open bounded neighborhood of the origin in an orthogonal $G$-representation $V$ with the identity operator $\id:V \rightarrow V$, one has
		\begin{align*}
			\gdeg(\id,\Omega)=(G).
		\end{align*}
	\end{itemize}
 \vs
The following are additional properties of the map $\gdeg$ which can be derived from the four axiomatic properties defined above (cf. \cite{AED}, \cite{book-new}):		
\begin{itemize}
		\item[] {\textbf{(Multiplicativity)}} For any $(f_{1},\Omega
		_{1}),(f_{2},\Omega_{2})\in\mathcal{M} ^{G}$,
		\begin{align*}
			\gdeg(f_{1}\times f_{2},\Omega_{1}\times\Omega_{2})=
			\gdeg(f_{1},\Omega_{1})\cdot \gdeg(f_{2},\Omega_{2}),
		\end{align*}
		where the multiplication `$\cdot$' is taken in the Burnside ring $A(G )$.

		\item[] \textbf{(Recurrence Formula)} For an admissible $G$-pair
		$(f,\Omega)$, the $G$-degree \eqref{eq:G-deg0} can be computed using the
		following Recurrence Formula:
		\begin{equation}
			\label{eq:RF-0}n_{H}=\frac{\deg(f^{H},\Omega^{H})- \sum_{(K)>(H)}{n_{K}\,
					n(H,K)\, \left|  W(K)\right|  }}{\left|  W(H)\right|  },
		\end{equation}
		where $\left|  X\right|  $ stands for the number of elements in the set $X$
		and $\deg(f^{H},\Omega^{H})$ is the Brouwer degree of the map $f^{H}%
		:=f|_{V^{H}}$ on the set $\Omega^{H}\subset V^{H}$.
	\end{itemize}
\end{theorem}
\vs
\noi{\bf Computation of Brouwer equivariant degree.} 
We denote by $\{ \mathcal V_i \}_{i \in \mathbb{N}}$ the set of all irreducible $G$-representations and define the $i$-th basic degree as follows
\begin{align*}
\deg_{\mathcal{V}_{i}}:=\gdeg(-\id,B(\mathcal{V} _{i})).
\end{align*} 
Given any orthogonal $G$-representation with a $G$-isotypic decomposition
\[
V = \bigoplus_{i \in \mathbb{N}} V_i,
\]
and any $G$-equivariant  linear isomorphism $T:V\to V$, the Multiplicativity and Homotopy properties of the $G$-equivariant Brouwer degree, together with Schur's Lemma implies
\begin{align*}
%\label{eq:prod-prop}
  \gdeg(T,B(V))=\prod_{i \in \mathbb{N}} \gdeg
	(T_{i},B(V_{i}))= \prod_{i \in \mathbb{N}}\prod_{\mu\in\sigma_{-}(T)} \left(
	\deg_{\mathcal{V} _{i}}\right)  ^{m_{i}(\mu)}%  
\end{align*}
where $T_{i}=T|_{V_{i}}$ and $\sigma_{-}(T)$ denotes the real negative
spectrum of $T$. \vskip.3cm

Notice that each of the basic degrees: 
\begin{align*}
	\deg_{\mathcal{V} _{i}}=\sum_{(H)}n_{H}(H),
\end{align*}
can be practically computed, using the recurrence formula  \eqref{eq:RF-0}, as follows
\begin{align*}
n_{H}=\frac{(-1)^{\dim\mathcal{V} _{i}^{H}}- \sum_{H<K}{n_{K}\, n(H,K)\, \left|  W(K)\right|  }}{\left|  W(H)\right|  }.
\end{align*}
\vs
The following fact is well-known (see for example \cite{survey}).
\begin{lemma} \rm
    For any irreducible $G$-representation $\mathcal V$, the basic degree $\deg_{\mathcal V} \in A(G)$ is an involutive element of the Burnside Ring, i.e.
    \[
    (\deg_{\cV})^2=\deg_{\cV} \cdot \deg_{\cV}=(G).
    \]
\end{lemma}

\section{The $S^1$-Equivariant Degree} \label{sec:appendix_S1}
Let $W = \br^n$ be an orthogonal $S^1$-representation with an open bounded $S^1$-invariant set $\Om \subset \br \times W$ (we assume that $S^1$ acts trivially on $\br$). An $S^1$-equivariant map $f: \br \times W \rightarrow W$ is said to be {\it $\Om$-admissible} if $f^{-1}(0) \cap \partial \Om = \emptyset$, in which case the pair $(f,\Om)$ is called an {\it admissible $S^1$-pair in $W$}. We denote by $\mathcal M_1^{S^1}(\br \times W)$ the set of all admissible $S^1$-pairs in $W$ and by $\mathcal M_1^{S^1}$ the set of {\it all} admissible $S^1$-pairs, defined by taking a union over all orthogonal $S^1$-representations as follows
\[
\mathcal M_1^{S^1}:= \bigcup\limits_W \mathcal M_1^{S^1}(\br \times W).
\]
\begin{remark} \rm
    If $W$ is a complex $S^1$-representation, then any complex linear map $W \rightarrow W$ is $S^1$-equivariant. Moreover, for any $w \in W$ one has the following two possibilities
    \[
    G_w = \begin{cases}
        \bz_k \quad & \text{ for some } k \in \bz; \\
        S^1
    \end{cases}
    \]
\end{remark}
We denote by  $\Phi(S^1)$ the set of all subgroup conjugacy classes in $S^1$ and by $\Phi_n(S^1)$ the subgroup conjugacy classes in $S^1$ with $n$-dimensional Weyl groups, i.e.
\[
\Phi_n(S^1) := \{ (H) \in \Phi(S^1) : \dim(W(H)) = n \}.
\]
\begin{remark} \rm
The subgroup conjugacy classes are specified according to the dimension of their Weyl groups as follows
  \[
\Phi_n(S^1) := \begin{cases}
    (S^1) \quad & \text{ if } n = 0; \\
    \{ (\bz_k) \}_{k \in \bn} \quad & \text{ if } n = 1; \\
    \emptyset \quad & \text{ otherwise.}
\end{cases}
\]  
\end{remark}
We denote by $A_1(S^1)$ the free $\bz$-module generated by the set $\Phi_1(S^1) = \{ (\bz_k) \in \Phi_(S^1): k \in \bn \}$. Any element $a \in A_1(S^1) = \bz[\Phi_1(S^1)]$ can be represented by a formal sum
\[
a = \sum_{k \in \bn} n_k(\bz_k), \quad n_k \in \bz,
\]
where $n_k = 0$ for almost every $k \in \bn$. We are now in a position to present an axiomatic definition of the $S^1$-equivariant degree.
\begin{theorem} \rm
    There exists a function $\s1deg: \mathscr M_1^{S^1} \rightarrow A_1(S^1)$ satisfying the properties:
    \begin{enumerate}[label=($S_\arabic*$)]
        \item {\bf (Additivity)}: If for any two $S^1$-invariant, disjoint subsets $\Om_1,\Om_2 \subset \Om$, one has $f^{-1}(0) \cap \Om \subset \Om_1 \cup \Om_2$, then
        \[
        \s1deg(f,\Om) = \s1deg(f,\Om_1) + \s1deg(f,\Om_2).
        \]
        \item {\bf (Homotopy)}: If $h:[0,1] \times \br \times W \rightarrow W$ is an $\Om$-admissible $S^1$-equivariant homotopy, then
        \[
        \s1deg(h_0,\Om) = \s1deg(h_1,\Om),
        \]
        where $h_t: \br \times W \rightarrow W$ is used to indicate the map $h_t(\cdot, \cdot) := h(t, \cdot,\cdot)$.
        \item {\bf (Normalization)}: If $f$ is regular normal in $\Om$ and $f^{-1}(0) \cap \Om = G(w_0)$ for some $w_0 \in \Om$, then one has
        \[
        \s1deg(f,\Om) = \begin{cases}
            \rho_0 (\bz_l) \quad & \text{if } G_{w_0} = \bz_l \text{ for some } l \in \bn; \\
             0  \quad & \text{if } G_{w_0} = S^1, 
        \end{cases}
        \]
        where $\rho_0 := \operatorname{sign}\det(\left.Df(w_0)\right|_{S_{w_0}})$ (here $S_{w_0} \subset \br \times W$ denotes the slice to the orbit $G(w_0)$ at $w_0$, i.e. the orthogonal subspace in $\br \times W$ to the orbit $G(w_0)$ at $w_0$).
    \end{enumerate}
\end{theorem}
There is a well-defined degree, called the {\it cumulative $S^1$-degree}, associated with the $S^1$-degree and whose role in the construction of the twisted equivariant degree is analogous to the role played by the local Brouwer degree in the equivariant Brouwer degree.
\begin{definition} \rm
The cumulative $S^1$-degree is the function $\deg_{S^1}:\mathscr M_1^{S^1} \rightarrow \bz$, given by
\[
\deg_{S^1}(f,\Om) := \sum_{s > 0} d_s, \quad d_s := \operatorname{coeff}^{(\bz_s)}(\s1deg(f,\Om)).
\]
\end{definition}
\begin{remark} \rm
    The cumulative $S^1$-degree is a total algebraic count of the non-trivial $S^1$ orbits in the solution set $f^{-1}(0) \cap \Om$.
\end{remark}
\section{The Twisted Equivariant Degree} \label{sec:appendix_twisted_deg}
Let $\Gamma$ be a compact Lie group and define the product group $G:= S^1 \times \Gamma$. Any closed subgroup in $G$
can be uniquely identified with a triple $(K,\varphi,l)$, called the {\it twisted decomposition} of that subgroup, and consisting of a subgroup $K \leq \Gamma$, a homomorphism $\varphi: \Gamma \rightarrow S^1$ and a number $l \in \bn \cup \{0\}$ as follows
\begin{align} \label{def:twisted_subgroups}
 K^{\varphi,l} := \{(z,\gamma) \in S^1 \times K: \varphi(\gamma) = z^l \}.
\end{align}
\begin{remark} \rm
In the case that $l = 0$, the triple $(K,\varphi,0)$ is always (that is, for any subgroup $K \leq \Gamma$ and homomorphism $\varphi: \Gamma \rightarrow S^1$) associated with
the {\it product subgroup}
\[
K^{\varphi,0} = S^1 \times K.
\]
\end{remark}
We denote by $\Phi_1^t(G)$ the set of {\it conjugacy classes} of twisted subgroups \eqref{def:twisted_subgroups} for which $W_\Gamma(K)$ is finite, i.e.
\begin{align} \label{def:twisted_subgroup_conjugacy_classes}
\Phi_1^t(G) := \{ (K^{\varphi,s}) \in \Phi(G) : \dim W_\Gamma(K) < \infty \},
\end{align}
and we define the free $\bz$-module generated by $\Phi_1^t(G)$ as follows
\[
A^t_1(G) := \bz[\Phi_1^t(G)].
\]
All elements $a \in A^t_1(G)$ are of the form
\begin{align}\label{eq:arbitrary_At1_element}
  a = \sum\limits_{(H) \in \Phi_1^t(G)} n_H(H),  \quad n_H \in \bz, 
\end{align}
where $n_H = 0$ except for a finite number of conjugacy classes $(H) \in  \Phi_1^t(G)$. The coefficient of any particular twisted orbit type $(H) \in \Phi_1^t(G)$ in \eqref{eq:arbitrary_At1_element} is specified with the notation
\begin{align}
\label{def:coefficient_operator_notation} 
\operatorname{coeff}^H(a) := n_H.    \end{align}
There is a natural $A(\Gamma)$-module structure on $A^t_1(G)$ induced by the product $A(\Gamma) \cdot A^t_1(G) \rightarrow A^t_1(G)$ defined, for any pair of generators $((K),(H)) \in \Phi_0(\Gamma) \times \Phi_1^t(G)$, as follows
\begin{align} \label{def:module_product}
(K) \cdot (H^{\varphi,l}) := \sum\limits_{(L) \leq (K)} n_L (L^{\varphi,l}),
\end{align}
where $n_L \in \bz$ is given by the number of type $(L^{\varphi,l})$-type orbits in the $G$-space $\frac{G}{K \times S^1} \times \frac{G}{H^{\varphi,l}}$. More practically,
a multiplication table for the $A(\Gamma)$-module $A_1^t(G)$ can be computed using the recurrence formula
\begin{align} \label{def:recurrence_formula_module_product}
    n_L = \frac{n(L,K)|W_\Gamma(K)|n(L^{\varphi,l},H^{\varphi,l})|W(H^{\varphi,l})/S^1| - \sum_{(\tilde{L}) > (L)} n_{\tilde L} n(L^{\varphi,l},\tilde L^{\varphi,l}) | W(\tilde L^{\varphi,l})/S^1 |}{| W(L^{\varphi,l})/S^1 |}.
\end{align}
An {\it admissible $G$-pair} $(f, \Om)$ in an orthogonal $G$-representation $U$ consists of an open bounded $G$-invariant set $\Om \subset \br \times U$ and a continuous $G$-equivariant map  $f: \br \times U \rightarrow U$ with $f^{-1}(0) \cap \Om = \emptyset$. We denote by $\mathcal M_1^{G}(U)$ the set of all admissible $G$-pairs in $U$ and by $\mathcal M_1^{G}$ the set of {\it all} admissible $G$-pairs, defined by taking a union over all orthogonal $G$-representations as follows 
\[
\mathcal M_1^{G} := \bigcup_{U} \mathcal M_1^{G}(U).
\]
Whereas an axiomatic definition for the $S^1$-degree was provided in Section \ref{sec:appendix_S1}, we now present a closed-form definition for the twisted $G$-equivariant degree.
\begin{definition} \rm
    The {\it twisted $G$-equivariant degree} is a map $\gdeg: \mathcal M_1^G \rightarrow A_1^t(G)$ assigning to every admissible $G$-pair $(f, \Om) \in \mathcal M_1^G$ the $A_1^t(G)$ element
        \[
\gdeg(f,\Om) := \sum\limits_{(H) \in \Phi_1^t(G)} n_H (H),
\]
where the integer coefficients $n_H \in \bz$ are given by the recurrence formula 
\begin{align} \label{def:twisted_degree_recurrence}
n_H := \frac{\deg_{S^1}(f^H,\Om^H) - \sum_{(L)>(H)} n_L n(H,L) | W(L)/S^1|}{|W(H)/S^1|}.    
\end{align}
\end{definition}
\begin{remark} \rm
The twisted $G$-equivariant degree satisfies the standard degree properties
\begin{enumerate}[label=($T_\arabic*$)]
\item\label{t1} {\bf (Existence)} If for any $(H) \in \Phi_1^t(G)$, one has $\operatorname{coeff}^H(\gdeg(f,\Om))$, then there is a solution $x \in \Om$ to the equation $f(x) = 0$ with $(G_x) \geq (H)$.
\item\label{t2} {\bf (Additivity)} If, for any two $G$-invariant disjoint subsets $\Om_1,\Om_2 \subset \Om$, one has $f^{-1}(0) \cap \Om \subset \Om_1 \cup \Om_1$, then
\[
\gdeg(f,\Om) = \gdeg(f,\Om_1) + \gdeg(f,\Om_2).
\]
\item\label{t3} {\bf (Homotopy)} If $h: [0,1] \times \br \times V \rightarrow V$ is an $\Om$-admissible $G$-homotopy, then
\[
\gdeg(h_0,\Om) = \gdeg(h_1,\Om),
\]
where $h_t: \br \times V \rightarrow V$ is used to indicate the map $h_t(\cdot, \cdot) := h(t, \cdot,\cdot)$.
\end{enumerate}
\end{remark}
We require an additional property, linked to the product property of the Brouwer equivariant degree and which can be derived from the properties \ref{t1}--\ref{t3}, to effectively employ the twisted $G$-equivariant degree.
\begin{lemma} \rm \label{lemm:twisted_degree_product_property}
For any admissible $\Gamma$-pair $(\phi,D) \in \mathcal M^\Gamma$ and for any admissible $G$-pair $(f,\Om) \in \mathcal M^{G}_1$, one has $(f \times \phi, \Om \times D) \in \mathcal M^{G}_1$ and
\[
\gdeg(f \times \phi, \Om \times D) = \Gamma \text{\rm -deg}(\phi,D) \cdot \gdeg(f,\Om).
\]
\end{lemma}

\subsection{A Computational Formula for the Twisted $G$-Equivariant Leray-Schauder Degree} \label{sec:appendix_twisted_comp_form}
Let $V$ be a finite dimensional, orthogonal $\Gamma$-representation and let 
$\mathcal H$ be an isometric Banach $G$-representation of maps taking values in $V$. For each $m \in \bn$ we denote by $\mathcal W_m \simeq \bc$, the irreducible $S^1$-representation equipped with the {\it $m$-folding $S^1$-action}
\begin{align*} 
%\label{def:S1_action}
    e^{i \theta}w := e^{i m\theta} w, \quad \theta \in S^1, \; w \in \mathcal W_m,
\end{align*}
and by $\mathcal W_0 \simeq \br$, the irreducible $S^1$ representation on which $S^1$ acts trivially. Assuming that a complete list of irreducible $\Gamma$-representations $\{ \mathcal V_j \}_{j = 0}^r$ are made available, $\mathcal H$ has a
$G$-isotypic decomposition of the form
\begin{align} \label{eq:calg_iso_decomp}
    \mathcal H = \bigoplus_{m \geq 0} \bigoplus_{j = 0}^r  \mathcal H_{m,j},
\end{align}
where each {\it $G$-isotypic component} $\mathcal H_{m,j}$  is modeled on the irreducible $G$-representation
\begin{align} \label{eq:calg_irrep}
    \mathcal V_{m,j} := \mathcal W_s \otimes \mathcal V_j. 
\end{align}
Notice that, for $m = 0$, the irreducible $G$-representation $\mathcal V_{0,j}$ coincides with the irreducible $\Gamma$-representation $\mathcal V_j$. On the other hand, every irreducible $G$-representation $\mathcal V_{m,j}$ with $m > 0$ is associated with an admissible $G$-pair $(b_{m,j},\mathscr O_{m,j}) \in \mathcal M^{G}_1(\mathcal V_{m,j})$, called the {\it basic pair} on $\mathcal V_{m,j}$,  consisting of 
the $G$-invariant set,
\begin{align}  \label{def:basic_set}
\mathscr O_{m,j} := \{ (t,v) \in \br \times \mathcal V_{m,j} : \frac{1}{2} < \|v\|_{\mathcal H} < 2, \; |t| < 1 \},
\end{align}
and the $\mathscr O_{m,j}$-admissible $G$-equivariant map
\begin{align} \label{def:basic_map}
    b_{m,j}(t,v) := (1-\|v\|_{\mathcal H} + it) \cdot v.
\end{align}
The {\it basic twisted degree} on $\mathcal V_{m,j}$, denoted $\deg_{\mathcal V_{m,j}} \in A_1^t(G)$, is defined as the twisted $G$-equivariant degree
\begin{align} \label{def:basic_deg}
\deg_{\mathcal V_{m,j}} := \gdeg(b,\Om).
\end{align}
The following result can be derived from the recurrence formula \eqref{def:twisted_degree_recurrence}.
\begin{lemma} \rm
Given an irreducible  $G$-representation $\mathcal V_{m,j}$, one has
    \[
    \deg_{\mathcal V_{m,j}} = \sum\limits_{(H) \in \Phi_1^t(G)} n_H (H),
    \]
    where
    \[
    n_H := \frac{\frac{1}{2}\dim \mathcal V_{m,j}^H - \sum_{(L)>(H)} n_L n(H,L)|W(L)/S^1|}{|W(H)/S^1|}.
    \]
\end{lemma}
Let $a: S^1 \rightarrow GL^{G}(\mathcal H)$ be a continuous family of $G$-equivariant invertible linear operators 
We are interested in a computational formula for the twisted $G$-equivariant Leray-Schauder degree of an admissible $G$-pair $(T,\mathscr D) \in \mathcal M^{G}_1(\mathcal H)$ consisting of the $G$-invariant set
\[
\mathscr D := \{ (\lambda,v) \in \bc \times \mathcal H : \| v \|_{\mathcal H} < 2, \; \frac{1}{2} < | \lambda | < 2 \},
\]
and the $\mathscr D$-admissible $G$-equivariant operator $T: \bc \times \mathcal H \rightarrow \br \times \mathcal H$ given by
\[
T(\lambda,v) := \left(\theta(\lambda,v), a\left(\frac{\lambda}{|\lambda|}\right)v\right),
\]
where $\theta: \overline{\mathscr D} \rightarrow \br$ is any $G$-invariant function satisfying
\begin{align} 
\begin{cases}
\theta(\lambda,0) < 0 \quad & \text{ for } |\lambda | = 2 \\
\theta(\lambda,u) > 0  \quad & \text{ for } |\lambda | = \frac{1}{2}.
\end{cases}    
\end{align}
For convenience, we adopt the notations
\begin{align}\label{def:Dmj}
 \mathscr D_{m,j} := \{(\lambda,v) \in \mathscr D : v \in \mathcal V_{m,j} \},   
\end{align}
and
\begin{align}\label{def:Tmj}
T_{m,j}(\lambda,v) := (1 - |\lambda|, a_{m,j}(\lambda) \cdot v), \quad a_{m,j}:= a|_{\mathcal H_{m,j}} : \mathcal H_{m,j} \rightarrow \mathcal H_{m,j}.  
\end{align}
We are now in a position to present, without proof, the first of a set of two essential tools, the so-called Splitting Lemma, for this purpose.
\begin{lemma}\label{lemm:splitting_lemma} \rm
Using the notations \eqref{def:Dmj} and \eqref{def:Tmj} and for any $m,m' > 0$, $j,j' \in \{0,1,\ldots,r\}$, one has
\[
\gdeg(T_{m,j} \times T_{m',j'},\mathscr D_{m,j} \times \mathscr D_{m',j'}) =  \gdeg(T_{m,j}, \mathscr D_{m,j}) + \gdeg(T_{m',j'}, \mathscr D_{m',j'})
\]
\end{lemma}
The second tool essential to the computation of $\gdeg(T,\mathscr D)$ provides a means to relate the twisted $G$-equivariant Leray-Schauder degrees $\gdeg(T_{m,j}, \mathscr D_{m,j}) \in A_1^t(G)$ with the corresponding basic degree $\deg_{\mathcal V_{m,j}}$.
\begin{lemma} \rm
Using the notations \eqref{def:Dmj} and \eqref{def:Tmj} and for any $m > 0$, $j \in \{0,1,\ldots,r\}$, one has
    \[
    \gdeg( T_{m,j}, \mathscr D_{m,j}) = \deg(\det\nolimits_\bc(a_{m,j})) \deg_{\mathcal V_{m,j}}.
    \]
\end{lemma}

\end{document}